\documentclass[11pt]{article}
\usepackage{fullpage}
\usepackage{booktabs}
\usepackage{multirow}
\usepackage[utf8]{inputenc}
\usepackage{amsthm,amssymb}
\usepackage{amsmath}
\usepackage{tcolorbox}
\usepackage{thmtools}
\usepackage{mathtools}
\usepackage{subcaption,thm-restate}
\usepackage[mathcal,mathscr]{eucal}
\usepackage{epsfig,graphicx,graphics,color}
\usepackage{caption}
\usepackage{enumerate}
\usepackage{enumitem}
\usepackage[sort]{natbib}
 \bibpunct[, ]{(}{)}{,}{a}{}{,}%
\usepackage{titling}
\usepackage{hyperref}
\usepackage{verbatim}

\usepackage[margin=1in]{geometry}

\hypersetup{
    colorlinks=true,       
    linkcolor=blue,        
    citecolor=red,         
    filecolor=magenta,     
    urlcolor=cyan,         
    linktocpage=true
}
\usepackage{thm-restate}
\usepackage{cleveref}
\usepackage{float}
\usepackage{algpseudocode}
\usepackage{rotating}
\usepackage{makecell}
\usepackage{pdfpages}
\usepackage{graphicx}
\usepackage{subcaption}
\usepackage{algorithm}

\usepackage{longtable}
\usepackage{booktabs,siunitx,array,threeparttable}
\usepackage{pdflscape}
\usepackage[nolist]{acronym}

\algnewcommand{\Input}[1]{\State \textbf{Input:} #1}
\algnewcommand{\Output}[1]{\State \textbf{Output:} #1}

\acrodef{PTSP}[PTSP]{Probabilistic Traveling Salesman Problem}
\acrodef{TSP}[TSP]{Traveling Salesman Problem}
\acrodef{OP}[OP]{Orienteering Problem}
\acrodef{TSPSC}[TSP-SC]{Traveling Salesman Problem with Service Commitments}
\acrodef{BSH}[BSH]{Adaptive Binomial Sampling Heuristic}
\acrodef{LKH}[LKH]{Lin--Kernighan Heuristic}

\theoremstyle{plain}
\newtheorem{theorem}{Theorem}
\newtheorem{lemma}[theorem]{Lemma}

\theoremstyle{definition}

\renewenvironment{proof}[1]{\par\noindent{\itshape #1}\space\ignorespaces}{\par\medskip}
\newcommand{\Halmos}{\ensuremath{\quad\square}}
\newcommand{\SingleSpacedXII}{}
\newcommand{\ACKNOWLEDGMENT}[1]{\subsection*{Acknowledgments}\noindent #1}

\DeclareMathOperator*{\argmax}{arg\,max}

\def\final{1}  
\def\iflong{\iffalse}
\ifnum\final=0  
\newcommand{\mnote}[1]{{\color{orange}[{\tiny \textbf{Malte:} \bf #1}]\marginpar{\color{orange}*}}}
\newcommand{\jnote}[1]{{\color{red}[{\tiny \textbf{Julian:} \bf #1}]\marginpar{\color{red}*}}}
\newcommand{\lnote}[1]{{\color{blue}[{\tiny \textbf{Lina:} \bf #1}]\marginpar{\color{blue}*}}}
\newcommand{\anote}[1]{{\color{purple}[{\tiny \textbf{Arne:} \bf #1}]\marginpar{\color{purple}*}}}
\else 
\newcommand{\mnote}[1]{}
\newcommand{\jnote}[1]{}
\newcommand{\lnote}[1]{}
\newcommand{\anote}[1]{}
\fi
\usepackage{authblk}

\title{Optimal Service Commitments in Traveling Salesman Problems with Stochastic Demand}

\date{}

\author[1]{Lina Schmidt}
\author[1]{Julian Golak}
\author[1]{Arne Schulz}
\author[1]{Malte Fliedner}

\affil[1]{\footnotesize Institute of Operations Management, University of Hamburg Business School, Hamburg, Germany. \texttt{\{lina.schmidt, julian.golak, arne.schulz, malte.fliedner\}@uni-hamburg.de}}

\begin{document}
\maketitle
\begin{abstract}
\medskip
In ridepooling services, mobility providers must communicate an expected arrival time before customers decide whether to book. An ambitious commitment attracts higher demand than a conservative one, but is correspondingly harder to fulfill. Provider revenue is therefore governed by an endogenous interplay between the announced commitment and the demand it induces.

We study this interplay in a traveling salesman setting in which all customers share a shuttle to a common destination, formalizing it as the Traveling Salesman Problem with Service Commitments (TSP-SC), a two-stage stochastic optimization problem. In the first stage, the provider announces an arrival commitment to all customers interested in the trip, each of whom independently accepts or rejects the offer with a probability decreasing in the committed time. In the second stage, an Orienteering Problem determines a revenue-maximizing subset of accepting customers to be served within the commitment; the remaining customers are rejected.

We show that optimal commitments lie in a finite candidate set induced by the tour lengths of customer subsets, and develop two exact algorithms, a general and a faster variant for homogeneous customers, together with an Adaptive Binomial Sampling Heuristic (BSH) for larger instances. In a numerical study on instances with up to 100 customers, the heuristic deviates from the optimal expected revenue by at most 0.32~\% on average across all exactly solvable sizes. The results further reveal a striking operational regularity: as the customer base grows, the revenue-maximizing policy converges to a regime in which roughly one in three accepting customers is rejected, a rejection rate that remains stable across instance sizes. This tradeoff proves markedly asymmetric: raising the service rate to approximately 80~\% costs only a few percent of optimal revenue, whereas near-perfect reliability requires sacrificing a disproportionate share of it.
\medskip


\end{abstract}

\section{Introduction}
	Dial-a-ride and ridepooling services have become an integral part of urban mobility. Providers such as MOIA, Via, and the shared-ride options of Uber and Lyft (See \url{https://www.moia.io}, \url{https://ridewithvia.com}, \url{https://www.uber.com}, and \url{https://www.lyft.com}, all last accessed July 10, 2026.) operate fleets of shared vehicles that dynamically collect and transport passengers. Thus far, the vast majority of articles dedicated to this application context focus on the efficient pooling of customer requests and the routing of shared vehicles. However, there is another central and closely linked operational challenge that has not gathered much attention: the provider must first communicate an expected arrival time, which --- together with the price of the trip --- is the customer's primary basis for accepting and rejecting the offered ride. Consequently, the quoted time directly shapes demand.
	This creates a fundamental tension. If the provider quotes an ambitious arrival time, more customers are inclined to accept, but the tighter schedule constrains the provider's ability to serve all of them, potentially forcing the rejection of customers. Conversely, a conservative commitment gives the provider more operational flexibility, yet fewer customers will find the offer attractive enough to accept. The provider's revenue, which depends on the number of customers actually served, is thus governed by an endogeneity between the service commitment and the resulting demand.
	
	We capture this endogeneity in a Traveling Salesman Problem in which a single shared vehicle, for instance, a bus or van, transports customers to a common destination. The provider announces an arrival commitment. Each customer independently accepts or rejects the offer with a probability that decreases in the announced arrival time. Once customer decisions are revealed, the provider determines a route that maximizes the revenue from served customers while respecting the committed arrival time. The goal is to find the commitment that maximizes expected revenue.
	
	\subsection{Related Literature}
	
	Our work connects three main strands of literature: demand management and endogenous demand; probabilistic combinatorial optimization; and the \ac{TSP} with profits and the \ac{OP}. We will summarize the main developments in the three subfields and their connections to our work in the following.
	
	\paragraph{Demand management and endogenous demand.}
	A central theme of this work is the endogeneity between a service provider's operational decisions and the resulting customer demand. A rich literature on this interplay has developed within queueing theory. \cite{naor1969regulation} was the first to consider strategic customer behavior in queues, studying the balking decision of customers who observe the queue length before deciding whether to join. This line of research has expanded substantially since then. \cite{economou2021impact} addresses the problem from a game-theoretic perspective, distinguishing between observable and unobservable queues and deriving optimal information provision strategies from both the social welfare and the revenue maximization perspective.
	A particularly relevant strand within this literature examines how the selective disclosure of information influences customer decisions. \cite{simhon2016optimal} and \cite{li2021optimal} analyze when revealing the queue length maximizes throughput or social welfare, while \cite{lingenbrink2019optimal} show that revenue is maximized by a signal that reveals only whether the queue is long.
	These queueing models are conceptually related to the broader fields of Bayesian Persuasion and Information Design. In Bayesian Persuasion, an informed sender designs a signal to influence the action of a receiver who holds a prior belief \citep{kamenica2011bayesian, kamenica2019bayesian}. Information Design extends this framework to settings with multiple receivers \citep{bergemann2019information}.
	
	Our problem shares the core feature of these models in that a provider's decision, the service commitment, directly shapes customer behavior.
	However, the customers' decision to accept or reject the service depends solely on the announced arrival commitment, not on a prior belief that is updated by information provision. 
	Moreover, while the queueing literature typically studies the disclosure of an exogenous system state (queue length, service rate), our service commitment is an endogenous decision variable that simultaneously determines both the acceptance probability and the operational constraints.

    Demand management also arises in the routing literature when customers can be rejected. In a static setting, all requests are known and the provider selects which customers to serve \citep{schulz2024using}; in a dynamic setting, where demand is revealed over time, it can be beneficial to reject a customer now in order to serve more profitable customers later. \citet{fleckenstein2023recent} review demand management in vehicle routing. A closely related stream is time slot management in attended home delivery, where the provider decides which delivery time slots to offer or how to price them, anticipating that the offered slots shape customer demand while routing feasibility constrains what can be offered \citep{agatz2011time, klein2019differentiated}. Our application setting itself belongs to the family of dial-a-ride and ridepooling problems; we refer to \citet{ho2018survey} and \citet{molenbruch2017typology} for surveys of the dial-a-ride problem and to \citet{agatz2012optimization} for a review of dynamic ride-sharing. Our setting combines both levers: the announced service commitment steers the customers' acceptance decisions, and the subsequent tour decision determines which accepting customers are actually served.

	\paragraph{Probabilistic combinatorial optimization.}
	The stochastic demand structure of our problem is closely related to the \ac{PTSP}, introduced by \cite{jaillet1985probabilistic}. In the \ac{PTSP}, each node in a graph is present with a given probability, and the objective is to find an a priori tour of minimum expected length over all nodes; absent nodes are simply skipped. \cite{jaillet1985probabilistic} derives a closed-form expression for efficiently computing the expected tour length and provides exact and heuristic solution approaches. \cite{bertsimas1990priori} study a priori optimization more broadly and show that it achieves asymptotic performance comparable to re-optimization at a fraction of the computational cost. Further structural properties are discussed in \cite{bertsimas1993further}, and \cite{henchiri2014probabilistic} provide a survey. Stochastic customer demand has also been studied extensively for vehicle routing; see, e.g., \citet{secomandi2009reoptimization} for reoptimization approaches to the vehicle routing problem with stochastic demands.
	
	On the algorithmic side, \cite{shmoys2008constant} provide the first constant-factor approximation guarantees, which have since been improved in a series of works \citep{van2011deterministic, van2018priori, blauth2025improved}. Among heuristic approaches, \cite{bowler2003characterization} apply a stochastic annealing method, and \cite{campbell2008probabilistic} introduce the \ac{PTSP} with deadlines and present recourse and chance-constrained models for it.
	
	Our problem differs from the classical \ac{PTSP} in two key aspects. First, the node probabilities are not exogenous but are determined by the provider's arrival commitment, creating the endogeneity discussed above. Second, the objective is not to minimize tour length but to maximize the number of served customers subject to a tour length constraint.
	
	\paragraph{TSP with profits and the Orienteering Problem.}
	Once the arrival commitment is fixed and customer decisions are revealed, the provider's second-stage recourse problem is to select and route a subset of customers that can be served within the committed time. This links our work to the family of \acp{TSP} with profits \citep{feillet2005traveling}, and specifically to the \ac{OP} \citep{tsiligirides1984heuristic}, in which the objective is to maximize the collected reward subject to a travel time constraint. The problem is NP-hard \citep{golden1987orienteering}, and \cite{blum2007approximation} show that it is APX-hard, ruling out a polynomial-time approximation scheme unless $\textsf{P} = \textsf{NP}$. Exact solution approaches include the branch-and-cut methods of \cite{fischetti1998solving} and \cite{kobeaga2024revisited}. Heuristic methods are proposed, for instance, in \cite{golden1987orienteering} and \cite{kobeaga2018efficient}. Upper bounds on the optimal objective value are studied by \cite{leifer1994strong} and \cite{millar1997time}. Comprehensive surveys of the \ac{OP} literature can be found in \cite{vansteenwegen2011orienteering} and \cite{gunawan2016orienteering}. In our problem, the \ac{OP} arises as the second-stage recourse problem, with the key difference that its time budget is not exogenous but is itself the first-stage decision, which simultaneously determines the demand distribution.

	\subsection{Our Contribution}
	
	The streams reviewed above study how operational promises shape demand and how stochastic demand affects routing. In our problem, however, a single arrival commitment plays both roles at once: it determines each customer's acceptance probability and, at the same time, the time budget of the second-stage routing problem. To the best of our knowledge, this coupling has not been studied before, although it arises naturally whenever a dial-a-ride provider quotes an arrival time before customer decisions are known.
	
	We address this gap by introducing the \ac{TSPSC}, a two-stage stochastic optimization problem. In the first stage, the service provider announces an arrival commitment, upon which customers independently accept or reject the offer with probabilities determined by the commitment. In the second stage, the provider finds a revenue-maximizing route among the accepting customers that respects the committed time. Our contributions are as follows.
	
	\begin{itemize}
		\item We analyze the structure of the expected revenue as a function of the arrival commitment. We show that optimal solutions can only occur at a finite (though exponentially large) set of candidate points, which we exploit algorithmically. We further establish that computing the expected revenue for a given commitment is $\textsf{\#P}$-hard for general travel times and customer-specific revenues and acceptance probabilities.
		\item We develop two exact algorithms that trade off generality against runtime. The first is a brute-force approach with runtime $O(n\,2^{2n})$ that supports customer-specific revenues and acceptance probabilities. The second exploits structural properties of the problem, such as the triangle inequality and identical customer probabilities, to achieve an improved runtime of $O(n^2 2^n)$.
		\item We propose a heuristic that estimates the expected revenue via an adaptive sampling approach and identifies near-optimal commitments.
		\item We evaluate all algorithms in a comprehensive numerical study on instances with up to 100 customers. The sampling heuristic is near-optimal, with a mean relative optimality gap of at most 0.32\% on all instances solvable exactly ($n \leq 19$), while scaling far beyond the reach of the exact algorithms. The experiments also yield a managerial insight: as the customer base grows, the revenue-maximizing commitment settles into a regime in which the provider turns away roughly one in three accepting customers, a rejection rate that then remains stable across instance sizes. Moreover, the cost of reliability is asymmetric: the provider can raise the service rate to about 80\% at a revenue reduction of only about 4\%, whereas near-perfect service requires forgoing almost half of the optimal revenue.
	\end{itemize}

	\medskip
	The paper is organized as follows. Section~\ref{sec:problemDef} formally defines the \ac{TSPSC} and illustrates it with an example. Section~\ref{sec:structural} provides structural and complexity results: Section~\ref{sec:revenueFun} characterizes the expected revenue function, and Section~\ref{sec:complexity} establishes that evaluating it is $\textsf{\#P}$-hard. Section~\ref{sec:algorithms} presents two exact algorithms and  introduces the sampling-based heuristic. Section~\ref{sec:experiments} reports on the numerical study. Finally, Section~\ref{sec:conc} summarizes the paper and lists open problems for future research.
	
	\section{Traveling Salesman Problem with Service Commitments}
	\label{sec:problemDef}
	In this section, we introduce the \ac{TSPSC}. We formally define the problem in Section~\ref{sec:formalDef} and illustrate it with an example in Section~\ref{sec:example}. Throughout, $\mathbb{R}_+$ and $\mathbb{Z}_+$ denote the non-negative real numbers and integers, respectively, and we write $[t] \coloneqq \{1, \dots, t\}$ for a positive integer~$t$.
	
	\subsection{Problem Definition}\label{sec:formalDef}
	We consider the operational setting sketched in the introduction: customers have requested transportation to a shared destination, for instance, an airport, a train station, or a major event venue. A provider operating a single shared vehicle must announce an arrival time at the destination before customers decide whether to use the service. The earlier the announced arrival time, the more likely a customer is to accept the offer. Once the interested customers have placed their booking under the offered conditions, the provider is still free to either confirm or disconfirm their bookings. This is a necessary safeguard for the provider since the actual resource utilization is only known once customers have revealed their decisions. For each customer who accepts and is ultimately served, the provider receives a revenue.
	
	Formally, we define the set of customers as $A$ with $n \coloneqq |A|$. Each customer $a \in A$ is characterized by a revenue $r_a \in \mathbb{R}_+$ and a probability function $p_a \colon \mathbb{R}_+ \to [0,1]$, which is non-increasing and represents the probability of customer $a$ accepting the offer in response to an arrival commitment $x \in \mathbb{R}_+$. Customers decide independently and observe neither each other's decisions nor the system state; this corresponds to the independent activation model \citep[see, e.g.,][]{shmoys2008constant}. We treat the acceptance probability functions as given, e.g., estimated from historical offer-acceptance data. We assume that the provider may reject accepting customers at no cost. The operational consequences of this assumption are quantified in Section~\ref{subsec:Managerial}, and we discuss rejection penalties as an extension in Section~\ref{sec:conc}.
	
	The vehicle starts its tour at the customers' destination, which we refer to as the depot $s$. We assume throughout that the vehicle's seating capacity is not binding. Each service request specifies the customer's pick-up location. Together with the depot, these locations define a complete graph $G = (V, E)$, where $V = A \cup \{s\}$ and a function $d \colon E \to \mathbb{R}_+$ assigns a travel time to each edge. We assume that $d$ satisfies the triangle inequality, i.e., $d(u,w) \leq d(u,v) + d(v,w)$ for all $u,v,w \in V$. For any subset $S \subseteq A$, we denote by $\tau(S)$ the length of a shortest tour in $G$ visiting all customers in $S$ and returning to the depot, where $\tau(\emptyset) \coloneqq 0$. Since the vehicle departs from the depot at time zero and returns to it at the end of the tour, all served customers arrive at the destination together upon completion of the tour; the commitment $x$ is therefore kept precisely if the chosen tour has length at most $x$. Customers' acceptance decisions depend only on this arrival commitment; we abstract from pickup times and individual ride durations. Given an arrival commitment $x$, the probability that exactly the customers in $F \subseteq A$ accept the offer is
	\[
	\mathbb{P}(F \mid x) = \prod_{a \in F} p_a(x) \prod_{a \in A \setminus F} (1 - p_a(x)).
	\]
	
	The setting is static: all customers request service before the commitment is announced, decisions are made once, and accepting customers do not subsequently cancel. Because all customers travel to the same destination on the same vehicle, the provider announces a single arrival commitment to all customers. We formulate the problem as a two-stage stochastic optimization problem. In the first stage, the provider selects an arrival commitment $x \in \mathbb{R}_+$. In the second stage, after the set of accepting customers $F$ is realized, the provider determines a route that maximizes revenue while respecting the commitment. The overall objective is to find the commitment that maximizes expected revenue:
	\begin{equation}\label{eq:1stageOpt}
		\sup_{x \in \mathbb{R}_+} \mathbb{E}_{F\mid x}\left[ \text{rev}(x, F)\right].
	\end{equation}
	For a given realization $F$ and commitment $x$, the second-stage revenue is defined as
	\begin{equation}\label{eq:2stageOpt}
		\text{rev}(x, F) = \max \left\{ \sum_{a \in S} r_a \colon S \subseteq F,\; \tau(S) \leq x \right\},
	\end{equation}
	where the provider selects the most profitable subset of accepting customers that can be served within the committed arrival time; that is, $\text{rev}(x, F)$ is the optimal value of an \ac{OP} instance with time budget $x$. We refer to the two-stage problem \eqref{eq:1stageOpt}--\eqref{eq:2stageOpt} as the \ac{TSPSC}.
	
	\subsection{Example}\label{sec:example}
	
	We illustrate the \ac{TSPSC} with a small instance. Consider three customers $a_1$, $a_2$, $a_3$ and a depot $s$, with travel times $d(s,a_1)=2$, $d(s,a_2)=3$, $d(s,a_3)=4$, $d(a_1,a_2)=4$, $d(a_1,a_3)=5$, and $d(a_2,a_3)=6$. Each customer provides a revenue of $r_a = 1$, and the acceptance probability is $p_a(x) = \max(0, 1 - x/50)$. Table~\ref{tab:subset_tourlengths} lists all customer subsets together with their shortest tour lengths $\tau(S)$.
	We consider $x = 11$. The resulting acceptance probability is $p(11) = 0.78$. The probability $\mathbb{P}(F \mid x = 11)$ for each subset $F \subseteq A$ is listed in the third column of Table~\ref{tab:subset_tourlengths}.
	
	\begin{table}[h]
		\centering
		\begin{tabular}{c|c|c|c}
			$F$ & $\tau(F)$ & $\mathbb{P}(F\mid x=11)$ & $\text{rev}(11, F)$\\\hline
			$\emptyset$ & $0$ & $0.0106$ & $0$\\
			$\{a_1\}$ & $4$ & $0.0378$ & $1$\\
			$\{a_2\}$ & $6$ & $0.0378$ & $1$\\
			$\{a_3\}$ & $8$ & $0.0378$ & $1$\\
			$\{a_1,a_2\}$ & $9$ & $0.1338$ & $2$\\
			$\{a_1,a_3\}$ & $11$ & $0.1338$ & $2$\\
			$\{a_2,a_3\}$ & $13$ & $0.1338$ & $1$\\
			$\{a_1,a_2,a_3\}$ & $16$ & $0.4746$ & $2$\\
		\end{tabular}
		\caption{All customer subsets with their shortest tour length, probability of occurrence for $x = 11$, and attainable revenue for the example instance with $p(x) = \max(0, 1 - x/50)$.}
		\label{tab:subset_tourlengths}
	\end{table}
	
	The fourth column of Table~\ref{tab:subset_tourlengths} shows the second-stage revenue $\text{rev}(x, F)$ for each subset. If $\tau(F) \leq x$, all customers in $F$ can be served and $\text{rev}(x, F) = |F|$. Otherwise, the provider serves a largest subset of $F$ whose tour length does not exceed $x$. For instance, if $F = \{a_1, a_2, a_3\}$ with $\tau(F) = 16 > 11 = x$, serving all three customers is infeasible. The best feasible subsets are $\{a_1, a_2\}$ and $\{a_1, a_3\}$ with tour lengths $9$ and $11$, respectively, yielding $\text{rev}(11, F) = 2$. Similarly, for $F = \{a_2, a_3\}$ we have $\tau(F) = 13 > 11$, and since $\tau(\{a_2\}) = 6$ and $\tau(\{a_3\}) = 8$, at most one of the two accepting customers can be served, so $\text{rev}(11, F) = 1$. The expected revenue for $x = 11$ is then
	\begin{align*}
		\mathbb{E}_{F\mid x}\left[\text{rev}(11, F)\right] 
		&= \sum_{F \subseteq A} \text{rev}(11, F) \cdot \mathbb{P}(F \mid x = 11) \\
		&= 1.7316.
	\end{align*}
	Figure~\ref{fig:ExpectedRevenue} plots the expected revenue as a function of the arrival commitment for $x \in [0, 30]$. Its maximum is attained at $x = 16$.

	\begin{figure}[h]
		\centering
		\includegraphics[width=0.6\textwidth]{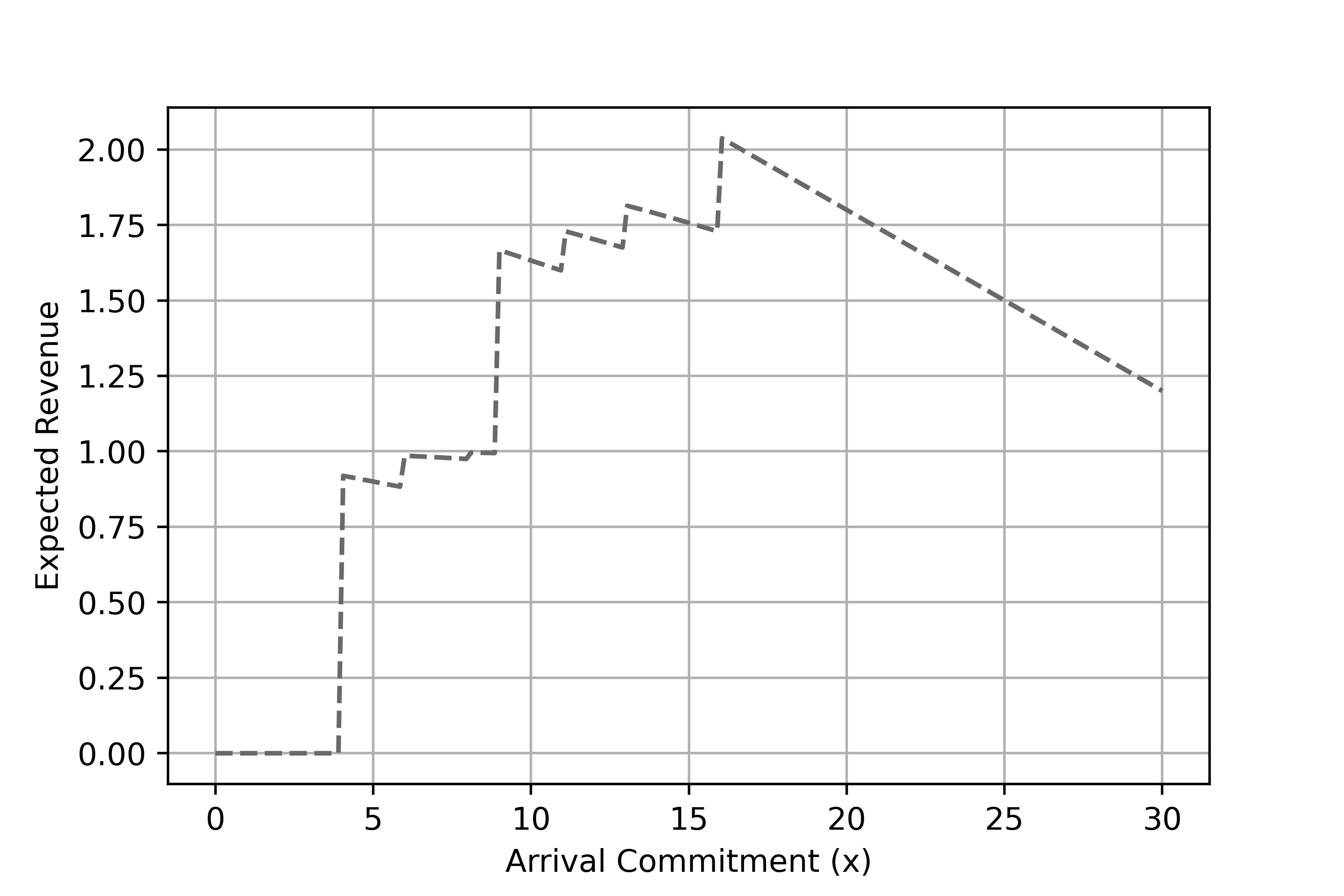}
		\caption{Expected revenue as a function of the arrival commitment $x$ for the example instance of Table~\ref{tab:subset_tourlengths}. The maximum is attained at $x = 16$.}
		\label{fig:ExpectedRevenue}
	\end{figure}
	
	\section{Structural Properties and Computational Complexity}
	\label{sec:structural}
	
	In this section, we present structural insights for the \ac{TSPSC} that are exploited in the subsequent algorithms. We begin by characterizing the expected revenue function in Section~\ref{sec:revenueFun}, and then show in Section~\ref{sec:complexity} that evaluating it for a given arrival commitment is $\textsf{\#P}$-hard.
	
	\subsection{Characterizing the Expected Revenue Function}
	\label{sec:revenueFun}
	We begin by characterizing the expected revenue as a function of the arrival commitment $x$. While the second-stage problem is a combinatorial optimization problem, the first-stage decision is a single continuous variable. We show that the set of candidate optimal solutions can be restricted to a finite set of discrete points, namely the tour lengths of the customer subsets.
	
	To make this precise, let $L_1 < L_2 < \cdots < L_K$ denote the elements of $\{\tau(S) \colon S \subseteq A\}$ in increasing order, and let $\mathcal{L} \coloneqq \{L_1, \ldots, L_K\}$, where $K \leq 2^n$. Since $L_1 = \tau(\emptyset) = 0$, setting $L_{K+1} \coloneqq \infty$ ensures that the intervals $[L_i, L_{i+1})$ for $i \in [K]$ partition $\mathbb{R}_+$. The following result shows that it suffices to evaluate the expected revenue at the tour lengths in $\mathcal{L}$.

	\begin{lemma} \label{lemma:discretisation_length}
		The supremum in \eqref{eq:1stageOpt} is attained at a tour length, i.e.,
		\[
			\sup_{x \in \mathbb{R}_+} \mathbb{E}_{F\mid x}\left[ \text{rev}(x, F) \right] = \max_{i \in [K]} \, \mathbb{E}_{F\mid L_i}\left[ \text{rev}(L_i, F) \right].
		\]
	\end{lemma}
	\begin{proof}{Proof.}
		We first show that $\mathbb{E}_{F\mid x}\left[ \text{rev}(x, F) \right]$ is non-increasing on each interval $[L_i, L_{i+1})$, $i \in [K]$. Fix $i \in [K]$ and consider any $x_1, x_2 \in [L_i, L_{i+1})$ with $x_1 \leq x_2$. We observe that $\text{rev}(x_1, F) = \text{rev}(x_2, F)$ for all $F \subseteq A$: since $\mathcal{L}$ contains all tour lengths, no tour length lies in the open interval $(L_i, L_{i+1})$, so a subset $S \subseteq A$ satisfies $\tau(S) \leq x_1$ if and only if $\tau(S) \leq x_2$, both conditions being equivalent to $\tau(S) \leq L_i$. Hence the feasible sets in \eqref{eq:2stageOpt} coincide for $x_1$ and $x_2$, and we write $\text{rev}(F)$ for the common value. Moreover, $\text{rev}$ is monotone with respect to set inclusion: if $F \subseteq F'$, then every $S \subseteq F$ with $\tau(S) \leq L_i$ is also a feasible subset of $F'$, so $\text{rev}(F) \leq \text{rev}(F')$.

		Let $(U_a)_{a \in A}$ be independent random variables, each uniformly distributed on $[0,1]$, and define the random sets $F(x_j) \coloneqq \{a \in A \colon U_a \leq p_a(x_j)\}$ for $j \in \{1,2\}$. By independence,
		\begin{equation}\label{eq:help1}
			\mathbb{P}\left(F(x_j) = F\right) = \prod_{a \in F} p_a(x_j) \prod_{a \in A \setminus F} (1 - p_a(x_j)) = \mathbb{P}(F \mid x_j),
		\end{equation}
		for all \(F \subseteq A\),
		so $F(x_j)$ is distributed according to $\mathbb{P}(\,\cdot \mid x_j)$. Since each $p_a$ is non-increasing, $U_a \leq p_a(x_2)$ implies $U_a \leq p_a(x_1)$, and hence $F(x_2) \subseteq F(x_1)$ holds with probability one. Thus, we have $\text{rev}(F(x_2)) \leq \text{rev}(F(x_1))$ pointwise, and taking expectations gives
		\[
			\mathbb{E}_{F\mid x_2}\left[ \text{rev}(x_2, F) \right] = \mathbb{E}\left[ \text{rev}(F(x_2)) \right] \leq \mathbb{E}\left[ \text{rev}(F(x_1)) \right] = \mathbb{E}_{F\mid x_1}\left[ \text{rev}(x_1, F) \right].
		\]
		Hence, the expected revenue is non-increasing on $[L_i, L_{i+1})$ for every $i \in [K]$. Since these intervals partition $\mathbb{R}_+$, the supremum over each of them is attained at its left endpoint $L_i$, and the claim follows.
	\Halmos
\end{proof}
	A complementary result holds when the acceptance probabilities are piecewise constant. Such probability functions arise naturally when customers respond to displayed time buckets rather than to exact arrival times---for instance, when the commitment is communicated as ``within 15, 30, or 45 minutes''---so that the acceptance probability changes only when the commitment crosses a bucket boundary. Formally, suppose there exist $B \in \mathbb{Z}_+$ and breakpoints $0 = b_0 < b_1 < \cdots < b_B = \infty$ such that $p_a(x) = c_i$ for all $a \in A$ and $x \in [b_{i-1}, b_i)$, $i \in [B]$. Within each such interval, the acceptance probabilities are fixed, and only the set of feasible tours changes with $x$.
	
	\begin{lemma} \label{lemma:discretisation_prob}
		The supremum in \eqref{eq:1stageOpt} is attained at one of at most $B$ candidate commitments, i.e.,
		\[
			\sup_{x \in \mathbb{R}_+} \mathbb{E}_{F\mid x}\left[ \text{rev}(x, F) \right] = \max_{i \in [B]} \, \mathbb{E}_{F\mid \ell_i}\left[ \text{rev}(\ell_i, F) \right],
		\]
		 where \(\ell_i \coloneqq \max\bigl(\left(\mathcal{L} \cap [b_{i-1}, b_i)\right) \cup \{b_{i-1}\}\bigr).\)
	\end{lemma}
	\begin{proof}{Proof.}
		We first show that $\mathbb{E}_{F\mid x}\left[ \text{rev}(x, F) \right]$ is non-decreasing on each interval $[b_{i-1}, b_i)$, $i \in [B]$. Let $x_1, x_2 \in [b_{i-1}, b_i)$ with $x_1 \leq x_2$. Since $p_a(x) = c_i$ for all $a \in A$ throughout the interval, we have $\mathbb{P}(F \mid x_1) = \mathbb{P}(F \mid x_2) = c_i^{|F|}(1-c_i)^{n-|F|}$ for all $F \subseteq A$. Furthermore, $x_1 \leq x_2$ implies $\text{rev}(x_1, F) \leq \text{rev}(x_2, F)$. Therefore,
		\begin{align*}
			\mathbb{E}_{F\mid x_1}\left[ \text{rev}(x_1, F) \right] - \mathbb{E}_{F\mid x_2}\left[\text{rev}(x_2, F) \right] 
			= \sum_{F \subseteq A} c_i^{|F|}(1-c_i)^{n-|F|} \left(\text{rev}(x_1, F) - \text{rev}(x_2, F)\right)
			\leq 0.
		\end{align*}
		Hence, the expected revenue is non-decreasing on $[b_{i-1}, b_i)$ for every $i \in [B]$. Moreover, since $\text{rev}(\cdot, F)$ changes only when $x$ crosses a tour length, the expected revenue is constant on $[\ell_i, b_i)$, so its supremum over $[b_{i-1}, b_i)$ is attained at $\ell_i$. Since the intervals $[b_{i-1}, b_i)$, $i \in [B]$, partition $\mathbb{R}_+$, the claim follows.
	\Halmos
\end{proof}

	In general, the expected revenue function is not unimodal, so the global maximizer cannot be identified by local search over the candidates in $\mathcal{L}$; this motivates evaluating all candidates in Section~\ref{sec:algorithms}. Moreover, the number of local maxima can grow with the number of customers. Consider $n = 7$ customers on a line with unit revenues and acceptance probability $p_a(x) = \max(0,\, 1 - x/50)$: three customers at distance $1.5$ from the depot, one at distance $8$, and three at distance $15$. Any subset of customers is served optimally by driving to the farthest requested location and back, so the candidate set is $\mathcal{L} = \{0, 3, 16, 30\}$, and for $x$ between consecutive candidates the expected revenue equals $p(x)$ times the number of customers within reach. As Figure~\ref{fig:multi_local_opt} shows, the function jumps upward at each tour length and declines in between, making every tour length a local maximum: it reaches $3\,p(3) = 2.82$ at $x = 3$, $4\,p(16) = 2.72$ at $x = 16$, and $7\,p(30) = 2.80$ at $x = 30$. Each peak reflects the tension between the additional revenue from newly reachable customers and the declining acceptance probabilities; appending further groups of customers at larger distances yields instances with arbitrarily many local maxima.
	
	\begin{figure}[htbp]
		\centering
		\includegraphics[width=0.6\textwidth]{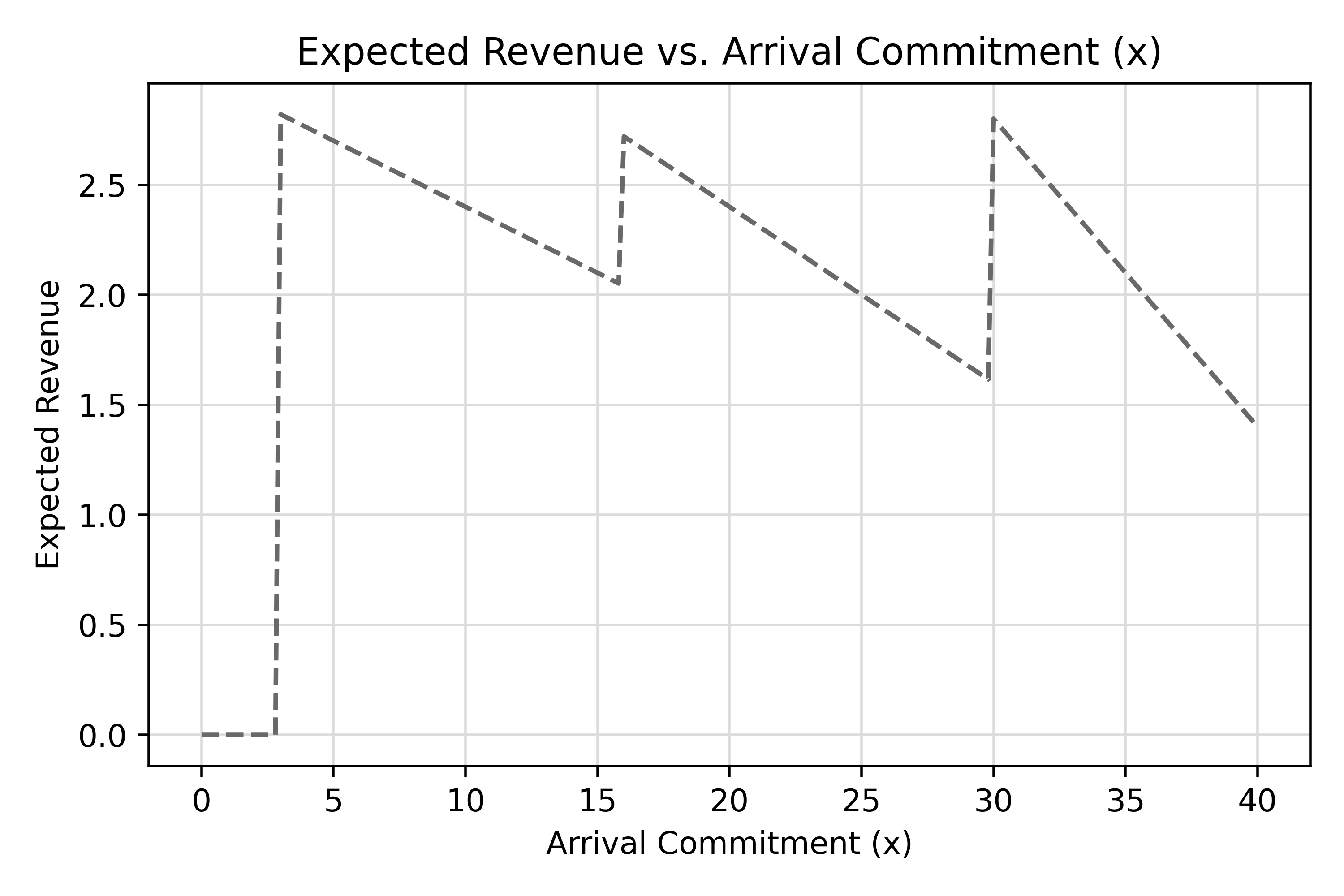}
		\caption{Expected revenue as a function of the arrival commitment $x$ for an instance with seven unit-revenue customers on a line (three at distance 1.5, one at distance 8, three at distance 15 from the depot) and $p_a(x) = \max(0,\, 1 - x/50)$. Each tour-length candidate is a local maximum.}
		\label{fig:multi_local_opt}
	\end{figure}

	\subsection{Hardness of Evaluating the Expected Revenue}\label{sec:complexity}
	In this subsection, we establish a complexity lower bound for evaluating the objective function of the first-stage problem. Specifically, we show that computing the expected revenue for a given arrival commitment is $\textsf{\#P}$-hard. We emphasize that this result applies exclusively to the evaluation of the expected revenue: the reduction relies on customer-specific revenues and acceptance probabilities as well as on travel times that violate the triangle inequality. Whether computing the expected revenue remains $\textsf{\#P}$-hard for metric travel times and homogeneous customers remains an open problem, as does the complexity of the first-stage problem.

	\begin{theorem} \label{thm:complexity}
		The problem of computing the expected revenue given an arrival commitment $x$ is $\textsf{\#P}$-hard for general edge weights.
	\end{theorem}

	\begin{proof}{Proof.}
	We show the claim by a reduction from the {\sc s-t Node Connectedness} problem, which was shown to be $\textsf{\#P}$-hard in the seminal work of \cite{valiant1979complexity}; a similar reduction was used by \citet{kamousi2011stochastic} to establish $\textsf{\#P}$-hardness of computing the expected minimum spanning tree length in stochastic networks. In this problem, we are given a graph $G'=(V',E')$ and a pair of vertices $s, t \in V'$, and the goal is to compute the number of subsets $W \subseteq V' \setminus \{s,t\}$ for which $G'$ contains a path between $s$ and $t$ whose internal vertices all belong to $W$.

	We construct an instance of our problem as follows. Let $n' \coloneqq |V'| - 1$. Introduce an auxiliary vertex $t'$ and set $V = V' \cup \{t'\}$, with $s$ as the depot and $A = V \setminus \{s\}$ as the set of customers, so that the constructed instance has $n = n' + 1$ customers. Define the edge set $\hat{E} = E' \cup \{\{s,t'\}, \{t,t'\}\}$. The travel times on the complete graph $G = (V, E)$ are defined by $d(e) = 0$ if $e \in \hat{E}$ and $d(e) = n+1$ otherwise; note that $d$ need not satisfy the triangle inequality. Set the acceptance probability of customer $t'$ to $p_{t'} \coloneqq 1$ and $p_a \coloneqq \frac{1}{2}$ for every other customer $a$, so that $t'$ accepts with probability one. Set the revenue of customer $t$ to $r_t = 1$ and the revenue of all other customers to $0$. We claim that evaluating the expected revenue of this instance at the arrival commitment $x = n$ solves the counting problem.

	With this construction, a tour of length at most $n$ can only use edges of $\hat{E}$; recall that a tour visits each served customer exactly once. Since $t$ is the only customer with positive revenue, $\text{rev}(n, F) = 1$ if some tour of length at most $n$ serves $t$, and $\text{rev}(n, F) = 0$ otherwise; in particular, $\text{rev}(n, F) = 0$ whenever $t \notin F$. Consider a realization $F$ with $\{t, t'\} \subseteq F$. If $G'$ contains a path $s = v_0, v_1, \ldots, v_k = t$ whose internal vertices $v_1, \ldots, v_{k-1}$ lie in $F$, then the tour $s \to v_1 \to \cdots \to t \to t' \to s$ serves the customers $\{v_1, \ldots, v_k, t'\} \subseteq F$, uses only edges in $\hat{E}$, and has length $0$, so $\text{rev}(n, F) = 1$. Conversely, suppose some tour of length at most $n$ serves $t$. If the tour visits $t'$, then, since the only edges of $\hat{E}$ incident to $t'$ are $\{s,t'\}$ and $\{t,t'\}$, the neighbors of $t'$ on the tour are $s$ and $t$, and removing $t'$ from the tour yields a path between $s$ and $t$ in $G'$ whose internal vertices lie in $F$; if the tour does not visit $t'$, the tour itself contains such a path. In both cases, $G'$ contains a path between $s$ and $t$ whose internal vertices lie in $F$. Hence, $\text{rev}(n, F) = 1$ if and only if $F = W \cup \{t, t'\}$ for some $W \subseteq V' \setminus \{s,t\}$ such that $G'$ contains a path between $s$ and $t$ whose internal vertices all belong to $W$.

	Since $t'$ accepts with probability one, exactly the realizations $F = W' \cup \{t'\}$ with $W' \subseteq V' \setminus \{s\}$ have positive probability, each occurring with probability $(\frac{1}{2})^{n'}$. We therefore obtain
	\begin{align*}
		2^{n'}\,\mathbb{E}_{F\mid x=n}\left[ \text{rev}(n, F )\right]
		&= \sum_{F \subseteq A \colon t' \in F} \text{rev}(n, F)\\
		&= \left|\left\{W \subseteq V' \setminus \{s,t\} \colon G' \text{ contains an $s$--$t$ path with all internal vertices in } W\right\}\right|.
	\end{align*}
	Thus, a single evaluation of the expected revenue at $x = n$, multiplied by $2^{n'}$, yields the answer to the counting problem, and computing the expected revenue is $\textsf{\#P}$-hard.
	\Halmos
\end{proof}
	We close this section with a remark that puts Theorem~\ref{thm:complexity} into perspective. The hardness concerns the exact evaluation of the expected revenue; approximating it is considerably easier, as sampling customer realizations yields an unbiased estimator of the expected revenue for any fixed commitment. The sampling-based heuristic developed in Section~\ref{sec:heuristic} builds on precisely this observation.

    \section{Exact and Heuristic Solution Algorithms}
    \label{sec:algorithms}
    In this section, we develop three algorithms for the \ac{TSPSC}. The first is a general brute-force approach that enumerates all candidate solutions independently, resulting in a runtime of $O(n\,2^{2n})$ (Section~\ref{sec:bruteForce}). This algorithm remains valid even if the triangle inequality does not hold and customers are heterogeneous, i.e., $r_a$ and $p_a$ may vary across customers. The second algorithm precomputes shared information across candidates, reducing the runtime to $O(n^2 2^n)$ (Section~\ref{sec:integrated}). Since even this reduced runtime remains intractable for larger instances, we develop a sampling-based heuristic that likewise relies on these two conditions, identifying near-optimal commitments efficiently (Section~\ref{sec:heuristic}). The second algorithm and the heuristic both require homogeneous customers and exploit the triangle inequality.
	
	\subsection{Brute-force Algorithm}\label{sec:bruteForce}
	
	By Lemma~\ref{lemma:discretisation_length}, the optimal expected revenue is attained at some tour length in $\mathcal{L}$. The brute-force approach computes $\tau(S)$ for every subset $S \subseteq A$, then evaluates the expected revenue at each candidate point and returns the maximum. 
	The algorithm is summarized in Algorithm~\ref{alg:bruteForce} and proceeds in two phases. Phase~1 (lines~\ref{line:bf:hk_start}--\ref{line:bf:hk_end}) computes the shortest tour length $\tau(F)$ for every subset $F \subseteq A$ using the seminal Held--Karp dynamic program \citep{held1962dynamic}, and collects the distinct tour lengths in the candidate set $\mathcal{L}$. Phase~2 (lines~\ref{line:bf:eval_start}--\ref{line:bf:eval_end}) iterates over each candidate $x \in \mathcal{L}$. For each candidate, the second-stage revenue $\text{rev}(x, F)$ is computed for all subsets via a dynamic program over subsets of increasing cardinality (lines~\ref{line:bf:dp_start}--\ref{line:bf:dp_end}): if $\tau(F) \leq x$, all customers in $F$ can be served; otherwise, the best achievable revenue is inherited from a subset of $F$ with one customer removed. The expected revenue is then obtained by weighting each $\text{rev}(x, F)$ by its probability (line~\ref{line:bf:exp_rev}), and the algorithm returns the candidate with the highest expected revenue.
	
	\begin{algorithm}[htbp]
		\SingleSpacedXII
		\caption{Brute-force algorithm for the \ac{TSPSC}}
		\label{alg:bruteForce}
		\begin{algorithmic}[1]
			\State \textbf{Input:} Graph $G = (V, E)$, probability functions $p_a$, revenues $r_a$
			\State \textbf{Output:} Optimal arrival commitment $x^*$ and expected revenue $R^*$
			\Statex
			\Comment{Phase 1: Compute all tour lengths via Held--Karp}
			\State $C(\{i\}, i) \gets d(s, i)$ for all $i \in A$ \label{line:bf:hk_start} \Comment{$C(S,i)$: shortest path from $s$ visiting all of $S$, ending at $i$}
			\For{$S \subseteq A$ with $|S| \geq 2$}
			\State $C(S, i) \gets \min_{j \in S \setminus \{i\}} \{ C(S \setminus \{i\}, j) + d(j, i) \}$ for all $i \in S$
			\EndFor
			\State $\tau(\emptyset) \gets 0$; $\tau(F) \gets \min_{i \in F} \{C(F, i) + d(i, s)\}$ for all nonempty $F \subseteq A$ \Comment{return to the depot}
			\State $\mathcal{L} \gets \{\tau(F) : F \subseteq A\}$ \label{line:bf:hk_end} \Comment{candidate commitments (Lemma~\ref{lemma:discretisation_length})}
			\Statex
			\Comment{Phase 2: Evaluate expected revenue at each candidate}
			\State $R^* \gets 0$ \label{line:bf:eval_start}
			\For{$x \in \mathcal{L}$}
			\State $\text{rev}(x, \emptyset) \gets 0$ \label{line:bf:dp_start}
			\For{nonempty $F \subseteq A$ in order of increasing $|F|$}
			\If{$\tau(F) \leq x$}
			\State $\text{rev}(x, F) \gets \sum_{a \in F} r_a$ \Comment{all accepting customers can be served}
			\Else
			\State $\text{rev}(x, F) \gets \max_{i \in F} \text{rev}(x, F \setminus \{i\})$ \Comment{omit one customer}
			\EndIf
			\EndFor \label{line:bf:dp_end}
			\State $R(x) \gets \sum_{F \subseteq A} \text{rev}(x, F) \cdot \mathbb{P}(F \mid x)$ \label{line:bf:exp_rev} \Comment{expected revenue at candidate $x$}
			\If{$R(x) > R^*$}
			\State $R^* \gets R(x)$
			\State $x^* \gets x$
			\EndIf
			\EndFor \label{line:bf:eval_end}
			\Statex
			\State \Return $x^*$, $R^*$
		\end{algorithmic}
	\end{algorithm}
	
	We now establish correctness and runtime.
	
	\begin{lemma}\label{lemma:naiveComputationLengths}
		Phase~1 computes $\tau(F)$ for all $F \subseteq A$ in $O(n^2 2^n)$ time and $O(n 2^n)$ space.
	\end{lemma}
	\begin{proof}{Proof.}
		The values $C(S, i)$ are computed bottom-up by increasing $|S|$ and stored 
		in a table, so each state is computed exactly once.
		The dynamic program has $O(n 2^n)$ states $(S, i)$, and each state requires $O(n)$ time to evaluate the minimization over $j \in S \setminus \{i\}$. Computing $\tau(F)$ from the table $C$ requires an additional $O(n)$ per subset. The total runtime is $O(n^2 2^n)$ and the space for the table $C$ is $O(n 2^n)$.
	\Halmos
\end{proof}
	
	\begin{lemma}\label{lemma:naiveExpectedRev}
		For a given $x \in \mathbb{R}_+$, Phase~2 computes $\mathbb{E}_{F\mid x}[\text{rev}(x, F)]$ in $O(n 2^n)$ time and  $O(2^n)$ space.
	\end{lemma}
	\begin{proof}{Proof.}
		The recurrence processes each $F \subseteq A$ in order of increasing $|F|$. 
		For each $F$, either $\tau(F) \leq x$ and the revenue is immediate, or we take the maximum over $|F|$ subproblems, each already computed.
		Since the subsets are processed in order of increasing cardinality and the 
		values $\text{rev}(x, F)$ are stored in a table, each subproblem is computed 
		exactly once.
		The total number of operations is $\sum_{F \subseteq A} |F| = O(n 2^n)$. The expected revenue is then obtained by a single pass over all $2^n$ subsets, requiring $O(2^n)$ space.
	\Halmos
\end{proof}
	
	\begin{theorem}\label{thm:opt1}
		Algorithm~\ref{alg:bruteForce} computes an optimal solution to the \ac{TSPSC} in $O(n\,2^{2n})$ time and $O(n 2^n)$ space.
	\end{theorem}
	\begin{proof}{Proof.}
		We first argue correctness. By Lemma~\ref{lemma:discretisation_length}, the optimal expected revenue is attained at some $x \in \mathcal{L}$. Phase~1 computes $\mathcal{L}$ exactly (Lemma~\ref{lemma:naiveComputationLengths}), Phase~2 evaluates the expected revenue at every candidate in $\mathcal{L}$ (Lemma~\ref{lemma:naiveExpectedRev}), and the algorithm returns a candidate with maximum expected revenue, which is therefore optimal.

		Regarding the runtime, Phase~1 requires $O(n^2 2^n)$ time by Lemma~\ref{lemma:naiveComputationLengths}, and Phase~2 evaluates each of the at most $2^n$ candidates in $O(n 2^n)$ time by Lemma~\ref{lemma:naiveExpectedRev}, for a total of $O(n\,2^{2n})$. Since the two phases run sequentially, the overall runtime is $O(n^2 2^n + n\,2^{2n}) = O(n\,2^{2n})$, where the final step uses $n \leq 2^n$. The space requirement is dominated by the Held--Karp table of Lemma~\ref{lemma:naiveComputationLengths}, which uses $O(n 2^n)$ space; the table of Phase~2 requires $O(2^n)$ space and is reused across candidates.
	\Halmos
\end{proof}

    \subsection{Integrated Algorithm}\label{sec:integrated}
	
	The brute-force algorithm recomputes the revenue for every subset at each candidate point, leading to the $O(n 2^{2n})$ runtime. We now present an algorithm that avoids this repeated work. It requires, however, that all customers share the same revenue $r_a = r$ and the same acceptance probability function $p_a = p$. Under these assumptions, we precompute an auxiliary table $\Gamma$ that summarizes, for each candidate tour length, how many subsets of each size can serve a given number of customers. Once $\Gamma$ is available, the expected revenue at any candidate point can be evaluated in $O(n^2)$ time, reducing the overall runtime to $O(n^2 2^n)$.
	
	Formally, we define $\Gamma(i, k, j)$ as the number of subsets $F \subseteq A$ with $|F| = k$ for which $\text{rev}(L_j, F) = i \cdot r$, where $r$ is the common revenue per customer. That is, $\Gamma(i, k, j)$ counts how many $k$-element subsets can serve exactly $i$ customers under arrival commitment $L_j$. The algorithm is summarized in Algorithm~\ref{alg:integrated} and proceeds in two phases. Phase~1 (lines~\ref{line:int:phase1_start}--\ref{line:int:phase1_end}) first computes all tour lengths via Held--Karp and sorts them into the candidate set $\mathcal{L}$. It then builds the table $\Gamma$ incrementally, maintaining for every subset $S \subseteq A$ a counter $\nu(S)$ that records the maximum number of customers of $S$ that can be served under the tour lengths processed so far: for each tour length $L_j$ in ascending order, it identifies the subsets $S$ with $\tau(S) = L_j$ and, processing them in order of non-decreasing cardinality, records that $S$ can now serve all $|S|$ of its customers (lines~\ref{line:int:update_start}--\ref{line:int:update_end}). The \textsc{Propagate} procedure (lines~\ref{line:int:prop_start}--\ref{line:int:prop_end}) then exploits the triangle inequality: since any superset $S' \supset S$ can serve at least as many customers as $S$, the table entries for all supersets with a lower current value are updated iteratively using a stack. By the triangle inequality, $\tau(S \setminus \{a\}) \leq \tau(S)$ for all $a \in S$, so when $S$ is processed, every proper subset of $S$ with the same tour length has already been processed and propagated, which guarantees $\nu(S) = |S| - 1$ at that moment. Phase~2 (lines~\ref{line:int:phase2_start}--\ref{line:int:phase2_end}) evaluates the expected revenue at each candidate $L_j$ in closed form using $\Gamma$ (line~\ref{line:int:exp_rev}), avoiding the exponential enumeration over subsets, and returns the candidate with the highest expected revenue.
	
	\begin{algorithm}[htbp]
		\SingleSpacedXII
		\caption{Integrated algorithm for the \ac{TSPSC}}
		\label{alg:integrated}
		\begin{algorithmic}[1]
			\State \textbf{Input:} Graph $G = (V, E)$, probability function $p$, revenue $r$
			\State \textbf{Output:} Optimal arrival commitment $x^*$ and expected revenue $R^*$
			\Statex
			\Comment{Phase 1: Compute all tour lengths and table $\Gamma$}
			\State Compute $\tau(F)$ for all $F \subseteq A$ via Held--Karp (as in Algorithm~\ref{alg:bruteForce}) \label{line:int:phase1_start}
			\State $\mathcal{L} = \{L_1, \ldots, L_K\} \gets$ distinct values of $\{\tau(F) : F \subseteq A\}$ in ascending order
			\State $\Gamma(i, k, 0) \gets 0$ for all $i \in \{0\} \cup [k]$, $k \in [n]$ \Comment{$\Gamma(0, \cdot, \cdot)$ is a dummy row, never read}
			\State $\nu(S) \gets 0$ for all $S \subseteq A$ \Comment{$\nu(S)$: customers of $S$ currently servable}
			\For{$j \in [K]$}
			\State $\Gamma(i, k, j) \gets \Gamma(i, k, j-1)$ for all $i \in \{0\} \cup [k]$, $k \in [n]$ \Comment{carry over previous counts}
			\For{all nonempty $S$ with $\tau(S) = L_j$, in order of non-decreasing $|S|$} \label{line:int:update_start}
			\State $\nu(S) \gets |S|$ \Comment{$S$ becomes fully servable}
			\State $\Gamma(|S|, |S|, j) \gets \Gamma(|S|, |S|, j) + 1$
			\State $\Gamma(|S|-1, |S|, j) \gets \Gamma(|S|-1, |S|, j) - 1$ \Comment{remove $S$ from its previous bucket}
			\State \Call{Propagate}{$S, |S|, j$}
			\EndFor \label{line:int:update_end}
			\EndFor \label{line:int:phase1_end}
			\Statex
			\Comment{Phase 2: Evaluate expected revenue at each candidate}
			\State $R^* \gets 0$ \label{line:int:phase2_start}
			\For{$j \in [K]$}
			\State $R(L_j) \gets \sum_{k=1}^{n} p(L_j)^k (1 - p(L_j))^{n-k} \sum_{i=1}^{k} i \cdot r \cdot \Gamma(i, k, j)$ \label{line:int:exp_rev}
			\If{$R(L_j) > R^*$}
			\State $R^* \gets R(L_j)$
			\State $x^* \gets L_j$
			\EndIf
			\EndFor
			\State \Return $x^*$, $R^*$ \label{line:int:phase2_end}
			\Statex
			\Function{Propagate}{$S, c, j$} \label{line:int:prop_start}
			\State $\mathcal{Q} \gets (S)$ \Comment{Initialize stack}
			\While{$\mathcal{Q} \neq \emptyset$}
			\State $\hat{S} \gets \text{pop from } \mathcal{Q}$
			\For{$a \in A \setminus \hat{S}$}
			\State $S' \gets \hat{S} \cup \{a\}$
			\If{$\nu(S') < c$} \Comment{superset $S'$ can now serve $c$ customers}
			\State $\Gamma(\nu(S'), |S'|, j) \gets \Gamma(\nu(S'), |S'|, j) - 1$
			\State $\nu(S') \gets c$
			\State $\Gamma(c, |S'|, j) \gets \Gamma(c, |S'|, j) + 1$
			\State push $S'$ onto $\mathcal{Q}$
			\EndIf
			\EndFor
			\EndWhile
			\EndFunction \label{line:int:prop_end}
		\end{algorithmic}
	\end{algorithm}
	
	We now establish correctness and runtime.
	
	\begin{lemma}\label{lemma:computeG}
		Phase~1 computes $\Gamma$ in $O(n^2 2^n)$ time and $O(n^2 K + n 2^n)$ space.
	\end{lemma}
	\begin{proof}{Proof.}
		The tour lengths are computed via Held--Karp in $O(n^2 2^n)$ time by Lemma~\ref{lemma:naiveComputationLengths}. It remains to bound the cost of computing $\Gamma$. For a given $j$, the \textsc{Propagate} procedure is called for each subset $S$ with $\tau(S) = L_j$, and iteratively propagates to supersets via a stack. For each subset $S \subseteq A$, the value $\nu(S)$ can increase at most $|S|$ times over the entire algorithm, and each subset popped from the stack requires iterating over at most $n - |S|$ elements of $A \setminus S$. Since the guard $\nu(S') < c$ ensures that each subset is enqueued at most once per value increase, each of the $\binom{n}{k}$ subsets of size $k$ is popped at most $k$ times, at a cost of $O(n)$ per pop. The total number of operations across all invocations is therefore bounded by
		\[
		\sum_{k=1}^{n} \binom{n}{k}\, k\, n = n^2 \sum_{k=1}^{n} \binom{n-1}{k-1} = n^2 \cdot 2^{n-1}.
		\]
		Copying $\Gamma(\cdot, \cdot, j-1)$ to $\Gamma(\cdot, \cdot, j)$ takes $O(n^2)$ per iteration, for a total of $O(n^2 K)$. The overall runtime is $O(n^2 2^n)$. The space is dominated by the table $\Gamma$ with $O(n^2 K)$ entries and the Held--Karp table with $O(n 2^n)$ entries.
	\Halmos
\end{proof}
	
	Phase~2 evaluates the expected revenue at each candidate using $\Gamma$, avoiding the exponential enumeration over subsets.
	
	\begin{lemma}\label{lemma:integratedExpectedRev}
		Given $\Gamma$ and an arrival commitment $L_j \in \mathcal{L}$, the expected revenue $\mathbb{E}_{F\mid L_j}[\text{rev}(L_j, F)]$ can be computed in $O(n^2)$ time.
	\end{lemma}
	\begin{proof}{Proof.}
		Since the acceptance probabilities are identical, the expected revenue can be decomposed by subset size:
		\begin{align*}
			\mathbb{E}_{F\mid L_j}\left[ \text{rev}(L_j, F)\right]
			&= \sum_{k=1}^{n} \mathbb{P}(|F| = k \mid L_j) \sum_{\{F \subseteq A : |F| = k\}} \mathbb{P}(F \mid L_j, |F| = k)\, \text{rev}(L_j, F) \\
			&= \sum_{k=1}^{n} p(L_j)^k (1 - p(L_j))^{n-k} \sum_{i=1}^{k} i \cdot r \cdot \Gamma(i, k, j),
		\end{align*}
		where we used $\mathbb{P}(|F| = k \mid x) = \binom{n}{k} p(x)^k (1-p(x))^{n-k}$ and $\mathbb{P}(F \mid x, |F| = k) = \binom{n}{k}^{-1}$. The expression is a double sum over $O(n^2)$ terms.
	\Halmos
\end{proof}
	
	\begin{theorem}\label{thm:opt2}
		Suppose all customers share the same revenue $r_a = r$ and the same acceptance probability function $p_a = p$. Then Algorithm~\ref{alg:integrated} computes an optimal solution to the \ac{TSPSC} in $O(n^2 2^n)$ time and $O(n^2 2^n)$ space.
	\end{theorem}
	\begin{proof}{Proof.}
		We first argue correctness. For $F \subseteq A$ and $j \in [K]$, let $\nu_j(F) \coloneqq \max\{|S| \colon S \subseteq F,\ \tau(S) \leq L_j\}$ denote the maximum number of customers of $F$ that can be served under commitment $L_j$; under identical revenues, $\text{rev}(L_j, F) = r \cdot \nu_j(F)$, and by the triangle inequality, $\nu_j$ is non-decreasing with respect to set inclusion. We claim that Phase~1 maintains the following invariant: after the iteration for $L_j$, we have $\nu(F) = \nu_j(F)$ for all $F \subseteq A$, and hence $\Gamma(i, k, j)$ equals the number of subsets $F \subseteq A$ with $|F| = k$ and $\nu_j(F) = i$. To see this, note that $\nu_j(F) > \nu_{j-1}(F)$ holds exactly if $F$ contains a subset $S$ with $\tau(S) = L_j$ and $|S| > \nu_{j-1}(F)$. The loop in lines~\ref{line:int:update_start}--\ref{line:int:update_end} processes every such $S$: since tied subsets are processed in order of non-decreasing cardinality, $\nu(S) = |S| - 1$ holds when $S$ is processed, so the update moves $S$ into bucket $|S|$ and the decrement removes it from bucket $|S| - 1$. \textsc{Propagate} then raises $\nu(S')$ to $|S|$ for every superset $S'$ of $S$ with a smaller current value and adjusts $\Gamma$ accordingly; every such superset is reached because, by monotonicity, all sets on an inclusion chain from $S$ to $S'$ also have values below $|S|$. Consequently, $\Gamma$ has the claimed semantics, Phase~2 evaluates the expected revenue exactly at every candidate by Lemma~\ref{lemma:integratedExpectedRev}, and, since the optimal expected revenue is attained at some $L_j \in \mathcal{L}$ by Lemma~\ref{lemma:discretisation_length}, the returned candidate is optimal.

		Regarding the runtime, Phase~1 requires $O(n^2 2^n)$ time by Lemma~\ref{lemma:computeG}, and Phase~2 evaluates each of the at most $2^n$ candidates in $O(n^2)$ time by Lemma~\ref{lemma:integratedExpectedRev}, for a total of $O(n^2 2^n)$. Since the two phases run sequentially, the overall runtime is $O(n^2 2^n)$. The space requirement is dominated by the table $\Gamma$ with $O(n^2 K) = O(n^2 2^n)$ entries, together with the $O(n 2^n)$ space of the Held--Karp computation by Lemma~\ref{lemma:computeG}.
	\Halmos
\end{proof}
    
	\subsection{Heuristic Algorithm}\label{sec:heuristic}
		
	Since the exact evaluation of the expected revenue requires enumerating all $2^n$ customer subsets, we propose a sampling-based heuristic that identifies near-optimal commitments efficiently. Throughout this section, we assume, as in Section~\ref{sec:integrated}, that all customers share the same revenue $r_a = r$ and the same acceptance probability function $p_a = p$.
	Section~\ref{sec:bsh} describes the overall algorithmic framework, which adaptively allocates samples across subset sizes.
	Section~\ref{sec:revenue_estimation} details the revenue estimation procedure, including partial-service recovery for infeasible subsets.
	Section~\ref{sec:estimation_error} discusses the two main sources of approximation error.
	\subsubsection{Adaptive Binomial Sampling Heuristic}\label{sec:bsh}
	
	The algorithm proceeds in three phases: sampling, revenue estimation, and an adaptive update step.
	For each subset size $k \in [n]$, let $N_k$ denote the number of sampled $k$-subsets, and let $T_k$ denote their tour lengths (computed via \ac{LKH} \citep{lin1973effective}), stored in ascending order.
	Samples accumulate across iterations: new tour lengths are appended to $T_k$, so the estimates improve as the algorithm progresses.
	The algorithm is summarized in Algorithm~\ref{alg:bsh}.
	
	\begin{algorithm}[htbp]
		\SingleSpacedXII
		\caption{Adaptive Binomial Sampling Heuristic (BSH)}
		\label{alg:bsh}
		\begin{algorithmic}[1]
			\State \textbf{Input:} Graph $G = (V, E)$, probability function $p$, revenue $r$, sample budget $N$, precision $\varepsilon > 0$, iteration limit $I_{\max}$
			\State \textbf{Output:} Estimated optimal commitment $\hat{x}$ and expected revenue $\hat{R}$
			\Statex
			\State $T_k \gets ()$ for all $k \in [n]$
			\State $\bar{p} \gets 0.5$ \label{line:bsh:p_init} \Comment{initial acceptance estimate}
			\State $\hat{x}_{\mathrm{prev}} \gets 0$
			\Statex
			\Comment{Phase 1: Sample tour lengths}
			\State $w_k \gets \binom{n}{k}\, \bar{p}^{\,k} (1-\bar{p})^{n-k}$ for all $k \in [n]$ \label{line:bsh:weights}\label{line:bsh:sample_start}
			\For{$k = 1, \ldots, n$}
			\State $N_k \gets \left\lfloor N \cdot w_k \big/ \sum_{j=1}^{n} w_j \right\rfloor$
			\State Draw $N_k$ random subsets of size $k$ and compute their tour lengths via \ac{LKH}; append to $T_k$
			\EndFor \label{line:bsh:sample_end}
			\Statex
			\Comment{Phase 2: Estimate revenue at each candidate}
			\State $\mathcal{X} \gets$ distinct values in $\bigcup_{k=1}^n T_k$ \label{line:bsh:candidates} \Comment{sampled candidate set}
			\State $\hat{x} \gets \argmax_{x \in \mathcal{X}}\; \Call{EstimateRevenue}{x, T_1, \ldots, T_n}$
			\State $\hat{R} \gets \Call{EstimateRevenue}{\hat{x}, T_1, \ldots, T_n}$
			\Statex
			\Comment{Phase 3: Update and check for convergence}
			\State \textbf{if} $|\hat{x} - \hat{x}_{\mathrm{prev}}| < \varepsilon$ \textbf{or} $I_{\max}$ iterations have been performed \textbf{then return} $\hat{x}$, $\hat{R}$ \label{line:bsh:terminate}
			\State $\hat{x}_{\mathrm{prev}} \gets \hat{x}$
			\State $\bar{p} \gets p(\hat{x})$ \label{line:bsh:alpha_update} \Comment{update the acceptance estimate}
			\State Go to line~\ref{line:bsh:sample_start} \label{line:bsh:iterate}
		\end{algorithmic}
	\end{algorithm}
	
	Phase~1 (lines~\ref{line:bsh:sample_start}--\ref{line:bsh:sample_end}) distributes the sample budget across subset sizes proportionally to the weights~$w_k$.
	In the first iteration, the weights correspond to $\mathrm{Binomial}(n, 0.5)$ (line~\ref{line:bsh:weights}), which is maximally spread across subset sizes.
	In practice, for small subset sizes where $N_k \geq \binom{n}{k}$, all subsets are enumerated exhaustively.
	Tour lengths are computed via \ac{LKH} and appended to the existing samples.
	Phase~2 selects the candidate commitment $\hat{x}$ that maximizes the estimated expected revenue over the sampled candidate set $\mathcal{X}$, using the \textsc{EstimateRevenue} function.
	Phase~3 (lines~\ref{line:bsh:terminate}--\ref{line:bsh:iterate}) first checks for termination: if the estimated commitment $\hat{x}$ changes by less than $\varepsilon$ between iterations or the iteration limit $I_{\max}$ is reached, the algorithm terminates.
	Otherwise, it updates the acceptance estimate to $p(\hat{x})$ and recomputes the sampling weights accordingly.
	This concentrates future samples on the subset sizes most likely to occur under the evolving estimate of the optimal commitment.
	
	\subsubsection{Revenue Estimation}\label{sec:revenue_estimation}
	
	The core of the heuristic is the revenue estimation procedure, which computes the estimated expected revenue at a given commitment $x$ without enumerating all subsets.
	For a given arrival commitment $x$, define the \emph{served share}
	\[
	\hat{f}_k(x) \coloneqq \frac{|\{t \in T_k : t \leq x\}|}{|T_k|},
	\]
	if \(T_k \neq \emptyset\) and \(\hat{f}_k(x) \coloneqq 0\) otherwise
	which estimates the fraction of $k$-subsets whose tour length does not exceed $x$; we collect these in the vector $\hat{f}(x) = (\hat{f}_1(x), \ldots, \hat{f}_n(x))$.
	The estimated expected revenue is
	\begin{equation}\label{eq:bsh_revenue}
		\hat{R}(x) = \sum_{k=1}^{n} \binom{n}{k} p(x)^k (1-p(x))^{n-k} \cdot \bigl(
		\hat{f}_k(x) \cdot k \cdot r
		+ (1 - \hat{f}_k(x)) \cdot \textsc{PartialRev}(\hat{f}(x), k, k-1)
		\bigr).
	\end{equation}
	The first term in the product is the probability that exactly $k$ of the $n$ customers accept the offer.
	The first term inside the parentheses accounts for the $k$-subsets that can serve all $k$ customers; we call these \emph{yes-instances}, and subsets whose full tour does not fit within the commitment \emph{no-instances}.
	The second term estimates the revenue from no-instances that may still serve a feasible subgroup, via the recursive procedure \textsc{PartialRev}.
	The complete estimation, including partial-service recovery, is summarized in Algorithm~\ref{alg:bsh_revenue}.
	
	\begin{algorithm}[htbp]
		\SingleSpacedXII
		\caption{Revenue estimation with partial-service recovery}
		\label{alg:bsh_revenue}
		\begin{algorithmic}[1]
			\Function{EstimateRevenue}{$x, T_1, \ldots, T_n$}
			\State $\hat{R}(x) \gets 0$
			\For{$k = 1, \ldots, n$} \label{line:rev:loop}
			\State $\hat{f}_k(x) \gets |\{t \in T_k : t \leq x\}| \big/ |T_k|$ if $T_k \neq \emptyset$, and $\hat{f}_k(x) \gets 0$ otherwise \label{line:rev:fhat}
			\State $\hat{R}(x) \gets \hat{R}(x) + \binom{n}{k} p(x)^k (1-p(x))^{n-k} \cdot \bigl( \hat{f}_k(x) \cdot k \cdot r + (1 - \hat{f}_k(x)) \cdot \Call{PartialRev}{\hat{f}(x), k, k-1} \bigr)$ \label{line:rev:add}
			\EndFor \label{line:rev:loop_end}
			\State \Return $\hat{R}(x)$
			\EndFunction
			\Statex
			\Function{PartialRev}{$\hat{f}(x), h, k$} \label{line:pr:start} \Comment{expected revenue of a no-instance of size $h$}
			\If{$k = 0$} \Return $0$ \EndIf \label{line:pr:base} \Comment{no feasible subgroup remains}
			\State $\tilde{f}_k \gets \hat{f}_k(x) \cdot (1 - \hat{f}_{k+1}(x))$ \label{line:pr:adjust} \Comment{overlap-adjusted served share}
			\State $g_k \gets 1 - (1 - \tilde{f}_k)^{\binom{h}{k}}$ \label{line:pr:combine} \Comment{probability that some $k$-subgroup is feasible}
			\State \Return $g_k \cdot k \cdot r + (1 - g_k) \cdot \Call{PartialRev}{\hat{f}(x), h, k-1}$ \label{line:pr:recurse}
			\EndFunction \label{line:pr:end}
		\end{algorithmic}
	\end{algorithm}
	
	The function iterates over subset sizes (lines~\ref{line:rev:loop}--\ref{line:rev:loop_end}).
	For each $k$, line~\ref{line:rev:fhat} computes the served share by binary search over the sorted tour lengths $T_k$.
	Line~\ref{line:rev:add} accumulates the weighted contribution: the binomial probability $\binom{n}{k} p(x)^k (1-p(x))^{n-k}$ that exactly $k$ customers accept, multiplied by the estimated revenue from both fully and partially served subsets.
	
	The \textsc{PartialRev} function (lines~\ref{line:pr:start}--\ref{line:pr:end}) estimates the expected revenue from a subset of size $h$ that cannot serve all $h$ customers, by recursing downward from $k = h - 1$.
	Line~\ref{line:pr:adjust} adjusts for overlap: the served share $\hat{f}_k(x)$ measures the fraction of all $k$-subsets that can be served, but some of these are subsets of $(k{+}1)$-subsets already counted as yes-instances at the level above; the adjustment $\tilde{f}_k = \hat{f}_k(x) \cdot (1 - \hat{f}_{k+1}(x))$ removes this overlap.
	Line~\ref{line:pr:combine} combines across subgroups: an $h$-subset contains $\binom{h}{k}$ subgroups of size $k$, and treating their feasibilities as independent, the probability that at least one is feasible is $g_k = 1 - (1 - \tilde{f}_k)^{\binom{h}{k}}$.
	Line~\ref{line:pr:recurse} recurses: with probability $g_k$, the subset serves $k$ customers; with the remaining probability $1 - g_k$, no $k$-subgroup is feasible, and the recursion continues to $k - 1$.
	
	\subsubsection{Sources of Estimation Error}\label{sec:estimation_error}
	The revenue estimator introduced in Section~\ref{sec:revenue_estimation} is subject to three sources of error: the approximation error of the tour-length oracle, the structural approximations in \textsc{PartialRev}, and the statistical error inherent in sampling.

	The \ac{LKH} does not guarantee optimal tour lengths.
	Let $\tau_{\mathrm{LKH}}(U)$ denote the tour length returned by \ac{LKH} for a customer subset $U \subseteq A$.
	Since $\tau_{\mathrm{LKH}}(U) \geq \tau(U)$, the served share $\hat{f}_k(x)$ computed from heuristic tour lengths is a lower bound on the true served share: subsets whose optimal tour fits within the commitment $x$ may be classified as infeasible if \ac{LKH} returns a suboptimal tour.
	This tends to depress the estimated revenue, although the partial-service recovery can partially offset the effect, so the net direction of the \ac{LKH}-induced bias is instance-dependent.
	In practice, \ac{LKH} produces near-optimal solutions for the instance sizes arising in ridepooling operations, so this bias is expected to be small.
	
	The \textsc{PartialRev} function introduces two simplifying assumptions.
	First, the overlap adjustment $\tilde{f}_k = \hat{f}_k(x) \cdot (1 - \hat{f}_{k+1}(x))$ in line~\ref{line:pr:adjust} treats the event that a $k$-subset is feasible as independent of whether it is contained in a feasible $(k+1)$-subset.
	In reality, if a $(k+1)$-subset has a short tour, its $k$-subsets are also likely to have short tours, so the adjustment tends to underestimate the fraction of $k$-subsets that contribute \emph{new} partial-service revenue.
	Second, the combination step in line~\ref{line:pr:combine} assumes that the feasibilities of distinct $k$-subgroups within the same $h$-subset are independent.
	Since subgroups sharing customers have positively correlated tour lengths, this overstates the probability that at least one $k$-subgroup is feasible.
	The two biases act in opposite directions: the overlap adjustment is conservative while the independence assumption is optimistic. Thus, the net effect is instance-dependent.

	Finally, the served shares $\hat{f}_k(x)$ are empirical frequencies based on $N_k$ samples and are therefore subject to statistical sampling error, which diminishes as the sample budget grows and as samples accumulate across iterations. A formal convergence analysis, in particular under the adaptive re-allocation of the sampling budget in Phase~3 of Algorithm~\ref{alg:bsh}, is beyond the scope of this paper.

 \section{Computational Study}
 \label{sec:experiments}
 This section evaluates the algorithms introduced in Section~\ref{sec:algorithms} in a computational study. We focus on three research questions.

 \textbf{Q1:} How well does the heuristic approximate the optimal expected revenue and arrival commitment, and how do its runtime and solution quality scale with the number of customers?

 \textbf{Q2:} What are the sources of the heuristic's estimation error, and how do they shape the estimated expected revenue?

 \textbf{Q3:} What operational regime does the revenue-maximizing commitment induce: how do the customer acceptance probability, the share of served customers, and the tightness of the commitment relative to the full-service tour behave across instance sizes, and at what revenue cost can the provider improve its service rate?
 
 All algorithms
 were implemented in Python~3.11. The computations were performed on the University of Hamburg's Hummel-2
 high-performance computing cluster (Debian GNU/Linux), whose compute nodes are equipped with two AMD EPYC~9654 processors (96 cores each) and 768~GB of RAM, using one core per instance.
 
 \subsection{Instance Generation and Algorithms}
 We generate 100 instances on a $15 \times 15$ kilometer grid representing a medium-sized city, with the depot located at the center and customer locations drawn uniformly at random. Each customer yields a unit revenue, $r = 1$. The acceptance probability is linear, $p(x) = \max(0,\, 1 - x/\text{slope})$, and, unless stated otherwise, we use slope $= 60$, meaning that customers reject any offer whose arrival commitment reaches one hour. Instances with $n \in \{2, 4, 6, 8, 10, 12, 15, 17, 19\}$ are referred to as \textit{small}; instances with $n \in \{20, 50, 100\}$ are referred to as \textit{large}. Sections~\ref{subsec:Performance}, \ref{subsec:ErrorSources}, and~\ref{subsec:Managerial} address Q1, Q2, and Q3 in turn, under this baseline slope; a sensitivity analysis of the acceptance-function slope is provided in Appendix~A.
 
 We evaluate three algorithms: the Held--Karp brute-force algorithm (HK, Algorithm~\ref{alg:bruteForce}) and the integrated algorithm (INT, Algorithm~\ref{alg:integrated}), which both compute the exact optimum, and the Adaptive Binomial Sampling Heuristic (BSH, Algorithm~\ref{alg:bsh}). The BSH obtains the tour length of a sampled customer subset either from the exact algorithms, which compute it individually (BSH-exact), or from the \ac{LKH} (BSH-\ac{LKH}). This yields three revenue measures: the optimum (HK/INT), BSH-exact, and BSH-\ac{LKH}. The research questions are addressed as follows. We answer Q1 with BSH-\ac{LKH}, the variant applicable at all instance sizes, benchmarked against the optimum on the small instances. We address Q2 by comparing BSH-exact with BSH-\ac{LKH}: since BSH-exact uses exact tour lengths, the gap between the two isolates the error introduced by the approximate tour-length computation. For the small instances, we additionally compare the revenue \emph{reported} by each BSH variant against the \emph{true} expected revenue of the commitment it returns, providing a direct measure of estimation accuracy. We answer Q3 with BSH-\ac{LKH} as well, characterizing the operational quantities at the commitment it returns across all instance sizes and reporting values based on the true expected revenue wherever the optimum is available.

 The heuristic uses a sample budget of $N = 30{,}000$ per iteration and terminates when the estimated commitment changes by less than $\varepsilon = 0.01$ between iterations or after 20 iterations, whichever occurs first.
 
 \subsection{Evaluation of Research Question Q1}\label{subsec:Performance}
  To answer Q1, we examine the runtime and solution quality reported in Tables~\ref{tab:full_results} and~\ref{tab:commitment_accuracy} and in Figure~\ref{fig:commitment_scatter_both}.
  
  Runtime grows sharply with $n$ for the exact algorithms. HK requires roughly $10{,}856$ seconds at $n = 15$ and could not be evaluated beyond this size. INT remains under $4$ seconds for all $n \leq 15$ and reaches $101$ seconds at $n = 19$, after which it fails at $n = 20$ due to memory exhaustion. The runtimes reported for BSH-exact are low because it reuses the tour lengths precomputed by HK and INT rather than computing them independently. BSH-\ac{LKH} has substantially higher runtimes at large $n$ ($10{,}763$ seconds at $n = 100$), reflecting the cost of repeated \ac{LKH} calls for every sampled subset. Notably, INT solves instances with up to 19 customers to optimality within seconds to about two minutes.
  
  Revenue increases monotonically with $n$ for all algorithms. For small instances we distinguish between the \emph{true} revenue, the exact expected revenue attained at the commitment returned by the heuristic, and the \emph{reported} revenue, the heuristic's own estimate of that value. For large instances, exact revenues are unavailable, so Table~\ref{tab:full_results} shows only BSH-\ac{LKH} reported revenues. Both BSH variants achieve near-optimal solution quality across all small instances, with a relative optimality gap of at most $0.32\%$ for all $n \leq 19$. The reported revenue continues to grow beyond the exactly solvable sizes, reaching a mean of $15.65$ at $n = 100$.
  
  The arrival commitment returned by both BSH variants closely tracks the optimal commitment. The mean commitment deviates from the optimum by at most about one time unit and almost always slightly exceeds it, with negligible impact on the attained revenue (see Table~\ref{tab:full_results}). At $n = 100$ the mean commitment reaches $47.0$. For $n \geq 8$, the heuristics identify the exact optimal commitment in $58\%$ to $77\%$ of instances, and in at least $84\%$ for smaller $n$ (Table~\ref{tab:commitment_accuracy}). When all instances are included, the relative deviation from the optimal commitment stays below $5\%$ across all $n$, with the largest value of $4.8\%$ at $n = 10$ for BSH-\ac{LKH}. Restricting attention to instances in which the optimum is missed, the relative deviation ranges from $20.2\%$ at $n = 6$ to $8.1\%$ at $n = 17$ for BSH-\ac{LKH}.
  
  This broadly decreasing trend is intuitive. First, for small instances the optimal commitment is short, so any deviation carries a larger relative weight. Second, the expected revenue function resembles a step function at small $n$, since the tour lengths of the few customer subsets are widely dispersed. A misplaced commitment therefore lands further from the optimum in relative terms. For larger instances the tour lengths are denser and the revenue function is smoother, so deviations have a smaller relative impact.
  
  Figure~\ref{fig:commitment_scatter_both} presents scatter plots comparing the commitment decisions of both heuristics against the optimal solution for $n \in \{6, 10, 15, 19\}$. Both algorithms cluster tightly around the diagonal across all instance sizes, and the scatter of the two methods is effectively equivalent. Point size encodes the number of coinciding observations. A consistent asymmetry is visible across all $n$. Overestimates are more dispersed and spread across a wider range, whereas underestimates concentrate at a few coordinates. Overestimation dominates in total frequency at $n = 6$ and $n = 19$, while at $n = 10$ and $n = 15$ underestimation is more frequent overall but the overestimated region remains more dispersed.

  In sum, BSH-\ac{LKH} attains near-optimal revenue and commitment decisions at all exactly verifiable sizes while scaling to instances far beyond the reach of the exact algorithms.

 \begin{landscape}
 	\begin{table}
 		\centering
 		\resizebox{\linewidth}{!}{%
 			\small
 			\begin{tabular}{r|rrrr|rrrrr|rrrrr}
 				\toprule
 				& \multicolumn{4}{c|}{\textbf{Optimal}} & \multicolumn{5}{c|}{\textbf{BSH-exact}} & \multicolumn{5}{c}{\textbf{BSH-LKH}} \\
 				$n$ & true rev. & arr.\ comm. & runtime HK & runtime INT & true rev. & report.\ rev. & \makecell{rel.\ gap} & arr.\ comm. & runtime & true rev. & report.\ rev. & \makecell{rel.\ gap} & arr.\ comm. & runtime \\
 				& \makecell{mean \\ (std)} & \makecell{mean \\ (std)} & \makecell{mean (std) \\ {[s]}} & \makecell{mean (std) \\ {[s]}} & \makecell{mean \\ (std)} & \makecell{mean \\ (std)} & abs. \ gap & \makecell{mean \\ (std)} & \makecell{mean (std) \\ {[s]}} & \makecell{mean \\ (std)} & \makecell{mean \\ (std)} & abs. \ gap & \makecell{mean \\ (std)} & \makecell{mean (std) \\ {[s]}} \\
 				\midrule
 				2 & \makecell{1.24\\ (0.23)} & \makecell{21.9\\ (6.1)} & \makecell{0.01\\ (0.00)} & \makecell{0.01\\ (0.00)} & \makecell{1.24\\ (0.23)} & \makecell{1.24\\ (0.23)} & \makecell{0.00\% \\ 0.0000} & \makecell{21.9\\ (6.1)} & \makecell{0.41\\ (0.09)} & \makecell{1.24\\ (0.23)} & \makecell{1.24\\ (0.23)} & \makecell{0.00\% \\ 0.0000} & \makecell{21.9\\ (6.1)} & \makecell{0.31\\ (0.06)} \\
 				4 & \makecell{1.95\\ (0.31)} & \makecell{25.5\\ (5.0)} & \makecell{0.01\\ (0.01)} & \makecell{0.01\\ (0.01)} & \makecell{1.95\\ (0.31)} & \makecell{1.93\\ (0.31)} & \makecell{0.18\% \\ 0.0035} & \makecell{26.0\\ (4.9)} & \makecell{0.40\\ (0.10)} & \makecell{1.95\\ (0.31)} & \makecell{1.92\\ (0.31)} & \makecell{0.25\% \\ 0.0047} & \makecell{25.9\\ (4.9)} & \makecell{0.34\\ (0.08)} \\
 				6 & \makecell{2.59\\ (0.36)} & \makecell{27.8\\ (4.5)} & \makecell{0.03\\ (0.01)} & \makecell{0.02\\ (0.01)} & \makecell{2.59\\ (0.36)} & \makecell{2.55\\ (0.36)} & \makecell{0.15\% \\ 0.0040} & \makecell{28.1\\ (4.3)} & \makecell{0.23\\ (0.08)} & \makecell{2.59\\ (0.36)} & \makecell{2.54\\ (0.36)} & \makecell{0.23\% \\ 0.0059} & \makecell{28.1\\ (4.3)} & \makecell{0.23\\ (0.05)} \\
 				8 & \makecell{3.15\\ (0.42)} & \makecell{29.0\\ (3.6)} & \makecell{0.31\\ (0.09)} & \makecell{0.01\\ (0.00)} & \makecell{3.15\\ (0.42)} & \makecell{3.13\\ (0.42)} & \makecell{0.19\% \\ 0.0055} & \makecell{29.2\\ (3.5)} & \makecell{0.16\\ (0.04)} & \makecell{3.15\\ (0.42)} & \makecell{3.12\\ (0.41)} & \makecell{0.25\% \\ 0.0074} & \makecell{29.2\\ (3.6)} & \makecell{0.34\\ (0.06)} \\
 				10 & \makecell{3.69\\ (0.42)} & \makecell{29.4\\ (3.8)} & \makecell{5.21\\ (0.49)} & \makecell{0.05\\ (0.01)} & \makecell{3.68\\ (0.42)} & \makecell{3.73\\ (0.43)} & \makecell{0.32\% \\ 0.0123} & \makecell{29.8\\ (3.1)} & \makecell{0.20\\ (0.02)} & \makecell{3.68\\ (0.42)} & \makecell{3.72\\ (0.43)} & \makecell{0.29\% \\ 0.0108} & \makecell{29.9\\ (3.1)} & \makecell{1.18\\ (0.25)} \\
 				12 & \makecell{4.25\\ (0.46)} & \makecell{30.1\\ (3.5)} & \makecell{116.96\\ (4.83)} & \makecell{0.33\\ (0.02)} & \makecell{4.24\\ (0.46)} & \makecell{4.37\\ (0.50)} & \makecell{0.21\% \\ 0.0091} & \makecell{30.2\\ (3.0)} & \makecell{0.56\\ (0.03)} & \makecell{4.24\\ (0.46)} & \makecell{4.36\\ (0.49)} & \makecell{0.24\% \\ 0.0103} & \makecell{30.2\\ (2.8)} & \makecell{7.00\\ (1.86)} \\
 				15 & \makecell{4.97\\ (0.52)} & \makecell{31.1\\ (3.0)} & \makecell{10856.06\\ (602.23)} & \makecell{3.87\\ (0.36)} & \makecell{4.96\\ (0.52)} & \makecell{5.23\\ (0.56)} & \makecell{0.21\% \\ 0.0108} & \makecell{31.0\\ (2.7)} & \makecell{8.81\\ (1.06)} & \makecell{4.96\\ (0.51)} & \makecell{5.21\\ (0.55)} & \makecell{0.22\% \\ 0.0115} & \makecell{31.2\\ (2.6)} & \makecell{99.87\\ (19.64)} \\
 				17 & \makecell{5.42\\ (0.56)} & \makecell{31.7\\ (3.2)} & -- & \makecell{20.36\\ (0.40)} & \makecell{5.41\\ (0.55)} & \makecell{5.74\\ (0.61)} & \makecell{0.22\% \\ 0.0126} & \makecell{32.1\\ (2.7)} & \makecell{16.08\\ (7.30)} & \makecell{5.41\\ (0.55)} & \makecell{5.72\\ (0.61)} & \makecell{0.22\% \\ 0.0124} & \makecell{32.0\\ (2.8)} & \makecell{296.42\\ (113.75)} \\
 				19 & \makecell{5.87\\ (0.54)} & \makecell{32.1\\ (2.7)} & -- & \makecell{101.41\\ (3.10)} & \makecell{5.85\\ (0.54)} & \makecell{6.24\\ (0.61)} & \makecell{0.24\% \\ 0.0141} & \makecell{33.0\\ (2.4)} & \makecell{21.26\\ (9.49)} & \makecell{5.85\\ (0.55)} & \makecell{6.21\\ (0.61)} & \makecell{0.24\% \\ 0.0142} & \makecell{33.1\\ (2.5)} & \makecell{365.11\\ (183.55)} \\
 				20 & -- & -- & -- & -- & -- & -- & -- & -- & -- & -- & \makecell{6.46\\ (0.62)} & -- & \makecell{33.4\\ (2.4)} & \makecell{418.44\\ (193.49)} \\
 				50 & -- & -- & -- & -- & -- & -- & -- & -- & -- & -- & \makecell{11.65\\ (0.88)} & -- & \makecell{41.2\\ (1.6)} & \makecell{2827.97\\ (1106.65)} \\
 				100 & -- & -- & -- & -- & -- & -- & -- & -- & -- & -- & \makecell{15.65\\ (0.93)} & -- & \makecell{47.0\\ (1.4)} & \makecell{10762.51\\ (4247.26)} \\
 				\bottomrule
 			\end{tabular}
 		}
 		\caption{Mean runtime, revenue, and arrival commitment by algorithm and instance size (slope $= 60$). The \emph{true} revenue denotes the exact expected revenue attained at the commitment returned by the heuristic. The \emph{reported} revenue is the heuristic's own estimate. The absolute gap is $|\text{optimal} - \text{true}|$ and the relative gap is $|\text{optimal} - \text{true}|/\text{optimal}$, both averaged across instances.}
 		\label{tab:full_results}
 	\end{table}
 \end{landscape}
 
 \begin{table}
 	\centering
 	\begin{tabular}{l ccc ccc}
 		\toprule
 		& \multicolumn{3}{c}{BSH-exact} & \multicolumn{3}{c}{BSH-LKH} \\
 		\cmidrule(lr){2-4} \cmidrule(lr){5-7}
 		$n$ & \% met & rel.\ dev.\ (all) & rel.\ dev.\ ($\neq$ opt) & \% met & rel.\ dev.\ (all) & rel.\ dev.\ ($\neq$ opt) \\
 		\midrule
 		$2$ & 100.0\% & 0.0\% & -- & 100.0\% & 0.0\% & -- \\
 		$4$ & 87.0\% & 2.4\% & 18.2\% & 85.0\% & 2.8\% & 18.5\% \\
 		$6$ & 86.0\% & 3.0\% & 21.5\% & 84.0\% & 3.2\% & 20.2\% \\
 		$8$ & 77.0\% & 2.9\% & 12.6\% & 77.0\% & 3.0\% & 13.0\% \\
 		$10$ & 72.0\% & 4.6\% & 16.3\% & 73.0\% & 4.8\% & 17.8\% \\
 		$12$ & 69.0\% & 3.7\% & 12.0\% & 70.0\% & 3.3\% & 11.1\% \\
 		$15$ & 66.0\% & 3.2\% & 9.5\% & 63.0\% & 3.4\% & 9.1\% \\
 		$17$ & 60.0\% & 3.7\% & 9.3\% & 58.0\% & 3.4\% & 8.1\% \\
 		$19$ & 62.0\% & 3.5\% & 9.3\% & 61.0\% & 3.6\% & 9.2\% \\
 		\bottomrule
 	\end{tabular}
 	\caption{Commitment accuracy of BSH-exact and BSH-LKH (slope $= 60$). Rel.\ dev.\ (all) is the mean
 		relative deviation of the heuristic's commitment from the optimal commitment over all instances;
 		rel.\ dev.\ ($\neq$ opt) restricts the mean to the instances in which the optimum is missed.}
 	\label{tab:commitment_accuracy}
 \end{table}

 \begin{figure}
 	\centering
 	\includegraphics[width=\textwidth]{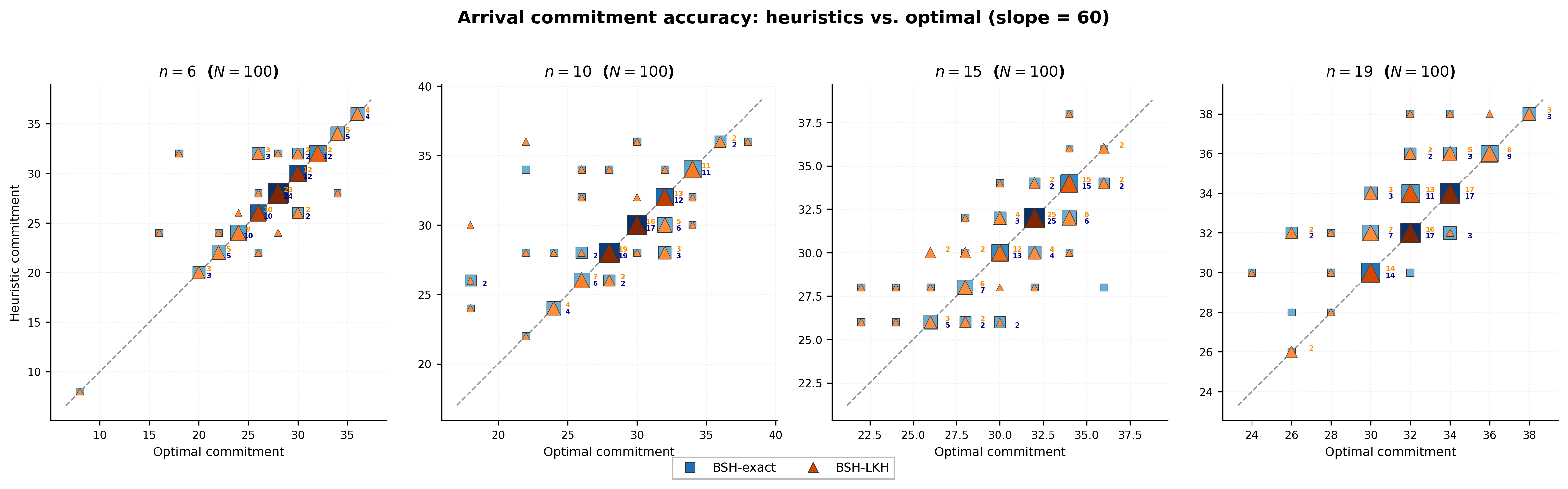}
 	\caption{Commitment decisions of BSH-exact and BSH-\ac{LKH} vs.\ optimal for \textit{small}
 		instances (slope $= 60$). Point size encodes the number of coinciding observations.}
 	\label{fig:commitment_scatter_both}
 \end{figure}

\subsection{Evaluation of Research Question Q2}\label{subsec:ErrorSources}
 Recall that BSH-exact uses exact tour lengths and therefore isolates the sampling error from the \ac{LKH} tour-length error; comparing it with BSH-\ac{LKH} reveals the contribution of the latter. This contribution is small: across all $n$, the true revenues of BSH-exact and BSH-\ac{LKH} are identical to two decimals, their reported revenues differ by at most $0.03$, and their optimality gaps by at most $0.08$ percentage points (Table~\ref{tab:full_results}); their commitment accuracies differ by at most three percentage points (Table~\ref{tab:commitment_accuracy}). The small remaining difference has the expected sign: since \ac{LKH} may return tours longer than optimal, some yes-instances are misclassified as infeasible, so the estimate of BSH-\ac{LKH} tends to lie slightly below that of BSH-exact (Figure~\ref{fig:revenue_function_instance}). The \ac{LKH} tour-length approximation therefore contributes negligibly to the overall error, which is instead dominated by the sampling-based \textsc{PartialRev} estimation.

 The estimation accuracy itself is high: the reported revenue closely tracks the true revenue for the smallest instances, where the sampling enumerates all customer subsets exhaustively. As $n$ grows, the reported revenue moderately overestimates the true revenue, reaching $6.21$ against a true value of $5.85$ at $n = 19$ for BSH-\ac{LKH}, consistent with the \textsc{PartialRev} overestimation discussed in Section~\ref{sec:estimation_error}. This bias remains moderate and does not translate into poor commitment decisions, as the optimality gaps in Table~\ref{tab:full_results} confirm.

 We now examine this dominant error, considering the expected revenue function for a single instance with $n = 15$ customers (Figure~\ref{fig:revenue_function_instance}) and averaged across all instances with $n = 15$ customers (Figure~\ref{fig:avg_revenue_function}). Both reveal a systematic over- and underestimation pattern: BSH \emph{underestimates} the expected revenue for shorter arrival commitments while \emph{overestimating} it for longer ones.
 
 This crossover is the signature of the two opposing approximations in \textsc{PartialRev} (Section~\ref{sec:estimation_error}). For short commitments, almost every subset is a no-instance, so nearly all estimated revenue flows through \textsc{PartialRev}, where the conservative overlap adjustment dominates and depresses the estimate. As the commitment grows, the optimistic independence assumption takes over and inflates the partial-service revenue assigned to the remaining no-instances, until enough subsets become genuine yes-instances and the effect fades.
 
 The crossover point in Figure~\ref{fig:revenue_function} marks where the two \textsc{PartialRev} biases offset one another; the instance-level and averaged curves are nearly indistinguishable between BSH-exact and BSH-\ac{LKH}.

 In sum, the heuristic's error stems almost entirely from the sampling-based estimation, whose two opposing approximations produce a predictable underestimation for short commitments and overestimation for long ones, while the tour-length approximation is negligible at every scale.

 \begin{figure}
 	\centering
 	
 	\begin{subfigure}[b]{0.48\textwidth}
 		\centering
 		\includegraphics[width=\textwidth]{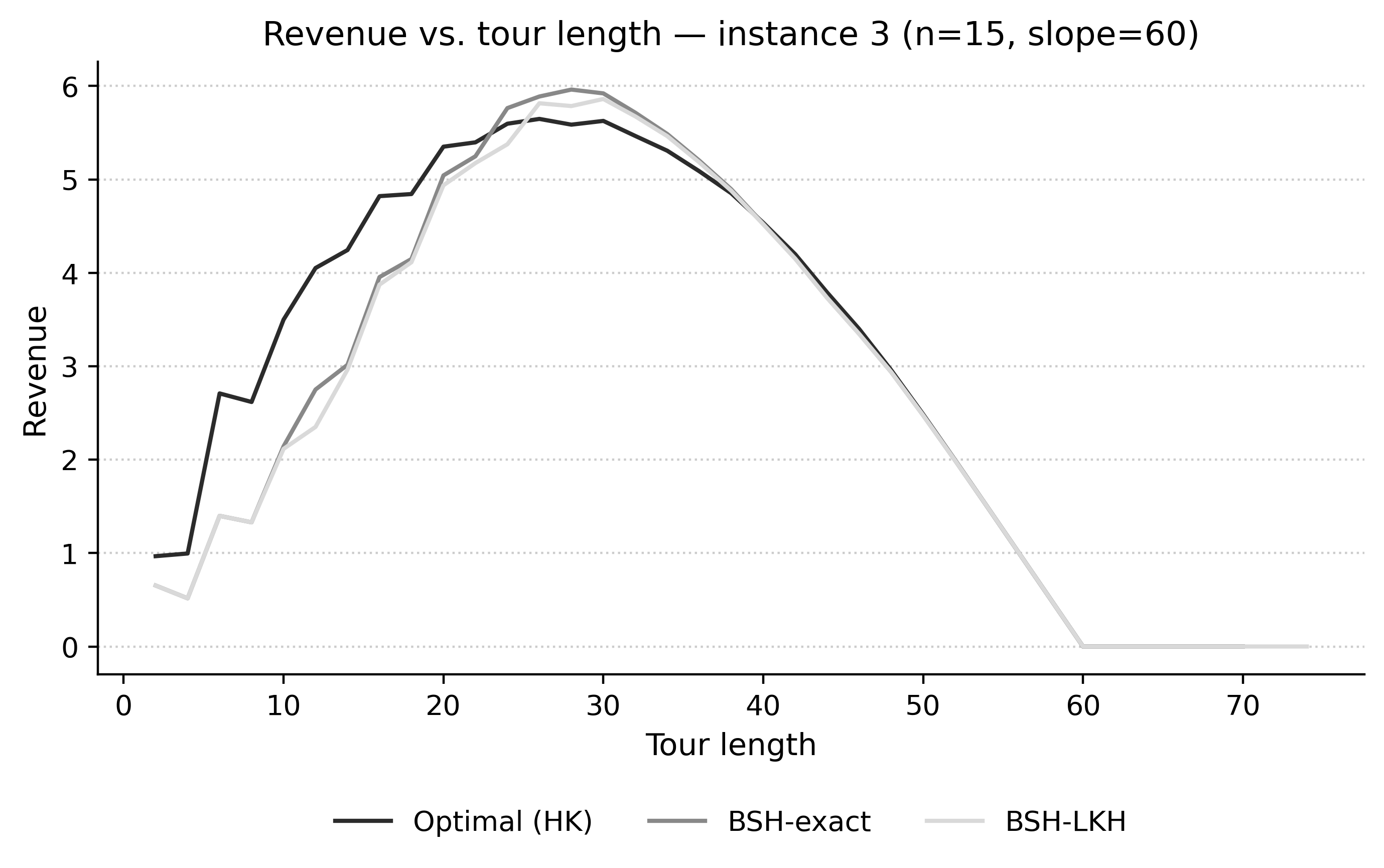}
 		\caption{Revenue functions for one instance.}
 		\label{fig:revenue_function_instance}
 	\end{subfigure}
 	\hfill
 	\begin{subfigure}[b]{0.48\textwidth}
 		\centering
 		\includegraphics[width=\textwidth]{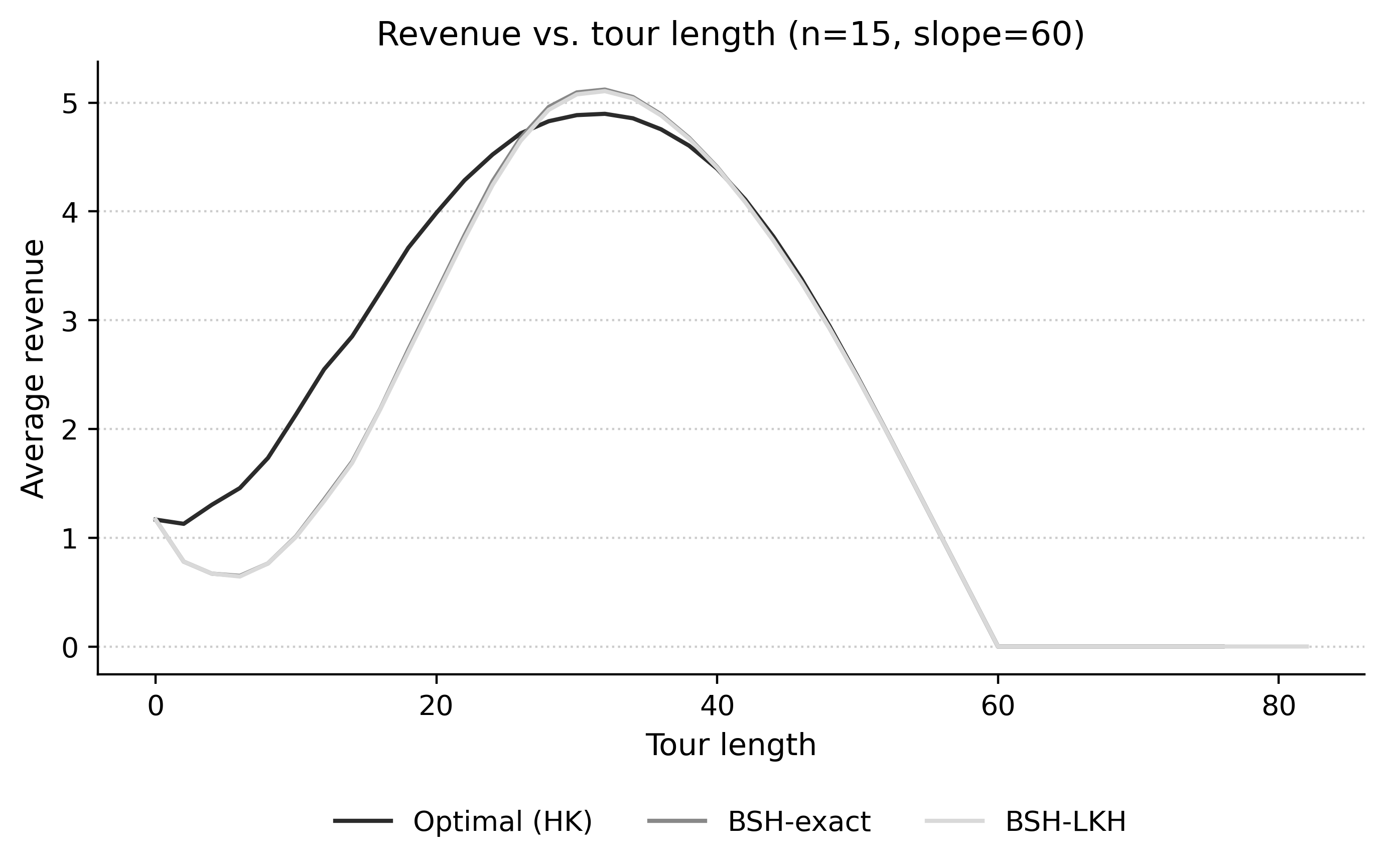}
 		\caption{Average revenue functions across all instances.}
 		\label{fig:avg_revenue_function}
 	\end{subfigure}
 	
 	\caption{True expected revenue (Optimal) and the reported revenue estimates of BSH-exact and BSH-\ac{LKH} as functions of the arrival commitment ($n = 15$, slope $= 60$).}
 	\label{fig:revenue_function}
 \end{figure}

\subsection{Evaluation of Research Question Q3}\label{subsec:Managerial}

 Table~\ref{tab:managerial_insights_summary} reports the operational quantities at the revenue-maximizing commitment by instance size. The commitment-tour ratio is the revenue-maximizing arrival commitment $x^*$ divided by $\tau(A)$, the length of the tour that serves \emph{all} customers. It measures what fraction of this full-service tour the committed time window covers: a ratio near one means the provider commits to nearly enough time to serve every customer, whereas a small ratio means the commitment covers only part of the full-service tour, so many accepting customers must be rejected. The ratio decreases from approximately $0.97$ at $n = 2$ to $0.30$ at $n = 100$: as the customer base grows, the revenue-maximizing strategy becomes increasingly conservative relative to what serving everyone would require. This is because the full-service tour length grows with $n$, whereas the commitment cannot usefully exceed the slope value of $60$, beyond which no customer accepts.
 
 The expected revenue $R^*$ alone does not reveal \emph{how} it is achieved: the same $R^*$ can arise from a conservative commitment that attracts few customers but serves most of them, or from an ambitious commitment that attracts many but serves only a fraction. To distinguish these regimes at the revenue-maximizing commitment $x^*$, we decompose the outcome into three quantities. First, we characterize customers' acceptance behavior. Second, we quantify the share of the customer population that the provider serves within the commitment (\emph{normalized revenue}). Third, we introduce the \emph{service rate}, which measures the probability that a customer is served conditional on having accepted the offer.
 
 The first two quantities are defined over the entire customer population. The \emph{customer acceptance probability} $p(x^*)$ is the probability that a random customer accepts the offer at $x^*$. The \emph{normalized revenue} is the probability that a random customer is served within the commitment, i.e., $R^*/n$. Both decline as $n$ grows. At $n = 2$ they nearly coincide, at $64\%$ and $62\%$, so system-side rejections are negligible. At $n = 100$, they diverge to $21.6\%$ and $15.7\%$. The widening gap reflects customers who accept the offer but are subsequently rejected by the system.
 
 We introduce the \emph{service rate} as the probability that a customer is served given that the customer accepted. Because a customer can only be served after accepting, this conditional probability equals the probability of being served divided by the probability of accepting---that is, the normalized revenue divided by the customer acceptance probability. The service rate is therefore the ratio of the two quantities reported in Table~\ref{tab:managerial_insights_summary}.
 
 The true service rate declines steadily, from $98\%$ at $n = 2$ to roughly $69\%$ by $n = 15$--$19$. The decline reflects the growing operational difficulty of serving all accepting customers as the customer base expands: the commitment covers a progressively smaller fraction of the full-service tour (cf.\ the declining commitment-tour ratio), so the provider must reject an increasing share of accepting customers to keep the commitment.
 
 From $n = 15$ onward, however, the revenue-maximizing policy settles into a stable regime in which about two thirds of accepting customers are served and roughly one in three is turned away. For the larger instances only the BSH-\ac{LKH} estimate is available; it suggests that this regime persists, although the \textsc{PartialRev} overestimation (Section~\ref{subsec:ErrorSources}) inflates the reported service rate, so the true value is likely a few points below the reported $72$--$74\%$. The stability itself is notable. As $n$ grows, both the customer acceptance probability and the normalized revenue continue to decline, yet the ratio between them converges. The provider attracts a smaller share of the population and serves a smaller share of the population, but it does so in consistent proportion. From the customer's perspective, the probability of being rejected after accepting remains roughly constant, largely independent of how large the customer base is. A provider who quotes the revenue-maximizing commitment must, in expectation, turn away roughly one in three customers who accepted the offer. This has implications beyond the single interaction modeled here. Repeated rejections after acceptance may erode customer trust and reduce future willingness to accept, an effect our model does not capture. A provider concerned with service reliability may therefore prefer a commitment slightly more conservative than $x^*$, accepting a modest reduction in expected revenue in exchange for a lower and more predictable rejection rate.
 
 
 Figure~\ref{fig:acceptance_probas_tour_lengths} makes the central managerial tradeoff
 explicit for $n = 15$ customers: the normalized expected revenue and the service rate peak
 at different commitments. The normalized revenue attains its maximum at the revenue-maximizing
 commitment $x^* = 32$. At $x^*$ it reaches approximately $33~\%$, but at
 that same commitment the service rate is only about $70~\%$ (the value in Table~\ref{tab:managerial_insights_summary}
 differs by roughly a point, reflecting the different averaging procedure discussed there). The service rate does not peak at $x^*$. It continues to climb as the commitment is
 extended, approaching $100\%$ once the committed time is long enough to serve essentially every
 accepting customer, precisely the region in which revenue has already fallen well below its
 maximum.
 
 Table~\ref{tab:managerial_insights_summary} and Figure~\ref{fig:acceptance_probas_tour_lengths}
 report related but distinct quantities, though both rest on the same exact revenue function
 $R(x)$. The table evaluates each instance's exact revenue at the commitment BSH-\ac{LKH}
 itself selects, then averages this outcome across instances. The figure instead averages
 the exact revenue curve across instances on a common grid and identifies the commitment
 $x^*$ that maximizes this averaged curve. The two therefore differ only in which $x$ is
 evaluated: the table reflects the heuristic's realized commitment choice, while the figure
 reflects the revenue-maximizing policy under exact optimization. This explains why the
 values do not coincide exactly.
 
 Although our objective includes no rejection penalty, the figure makes the tension such a
 penalty would capture explicit: the service rate rises above its value at $x^*$
 only by forgoing revenue. At the optimal commitment, the normalized revenue reaches
 $32.7~\%$ at a service rate of $70.0~\%$. The marginal cost of improving reliability beyond this point is markedly
 non-linear. Raising the service rate to $80~\%$ requires a revenue reduction of only
 $3.4~\%$. Reaching a service rate of $90~\%$ is considerably more costly, requiring a reduction
 of $13.9~\%$, and $95~\%$ requires foregoing roughly a quarter of the optimal
 revenue ($25.6~\%$). Beyond this point the tradeoff deteriorates sharply: eliminating
 nearly all rejections, a service rate of $99.2~\%$, requires a revenue reduction of
 $49.3~\%$, giving up nearly half of what was achievable at the revenue-maximizing commitment. These results indicate a natural operating
 region for providers who place additional weight on service reliability. The service rate can
 be raised substantially above its revenue-maximizing level at comparatively low cost up
 to approximately $80$--$90\%$, whereas pursuing a near-perfect service rate entails a
 disproportionate revenue sacrifice and is likely justified only when rejection carries a
 severe cost, for instance through reputational damage or regulatory requirements.
 
\begin{table}
	\centering
	\small
	\begin{tabular}{r r r rr rr}
		\toprule
		& & & \multicolumn{2}{c}{normalized revenue} & \multicolumn{2}{c}{service rate} \\
		\cmidrule(lr){4-5} \cmidrule(lr){6-7}
		customers & commitment/tour & cust.\ acceptance & reported & true & reported & true \\
		\midrule
		2 & 0.9724 & 0.6353 & 0.6219 & 0.6219 & 0.9788 & 0.9788 \\
		4 & 0.7602 & 0.5677 & 0.4809 & 0.4869 & 0.8472 & 0.8577 \\
		6 & 0.6677 & 0.5320 & 0.4240 & 0.4311 & 0.7970 & 0.8104 \\
		8 & 0.5890 & 0.5137 & 0.3905 & 0.3933 & 0.7603 & 0.7657 \\
		10 & 0.5430 & 0.5020 & 0.3719 & 0.3683 & 0.7408 & 0.7337 \\
		12 & 0.4966 & 0.4967 & 0.3631 & 0.3531 & 0.7310 & 0.7109 \\
		15 & 0.4571 & 0.4793 & 0.3471 & 0.3305 & 0.7241 & 0.6895 \\
		17 & 0.4404 & 0.4667 & 0.3363 & 0.3183 & 0.7206 & 0.6820 \\
		19 & 0.4369 & 0.4477 & 0.3269 & 0.3081 & 0.7301 & 0.6882 \\
		20 & 0.4318 & 0.4440 & 0.3231 & -- & 0.7277 & -- \\
		50 & 0.3659 & 0.3130 & 0.2330 & -- & 0.7443 & -- \\
		100 & 0.2956 & 0.2160 & 0.1565 & -- & 0.7245 & -- \\
		\bottomrule
	\end{tabular}
	\caption{Commitment-tour ratio, customer acceptance probability, normalized revenue, and service rate by instance size (slope $= 60$), evaluated at the commitment selected by BSH-\ac{LKH}. \emph{Reported} values of the normalized revenue and the service rate are based on the heuristic's revenue estimate, \emph{true} values on the exact expected revenue, which is computable only for the small instances ($n \leq 19$); a dash marks unavailable values.}
	\label{tab:managerial_insights_summary}
\end{table}
  
 \begin{figure}
 	\centering
 	\includegraphics[width=0.6\textwidth]{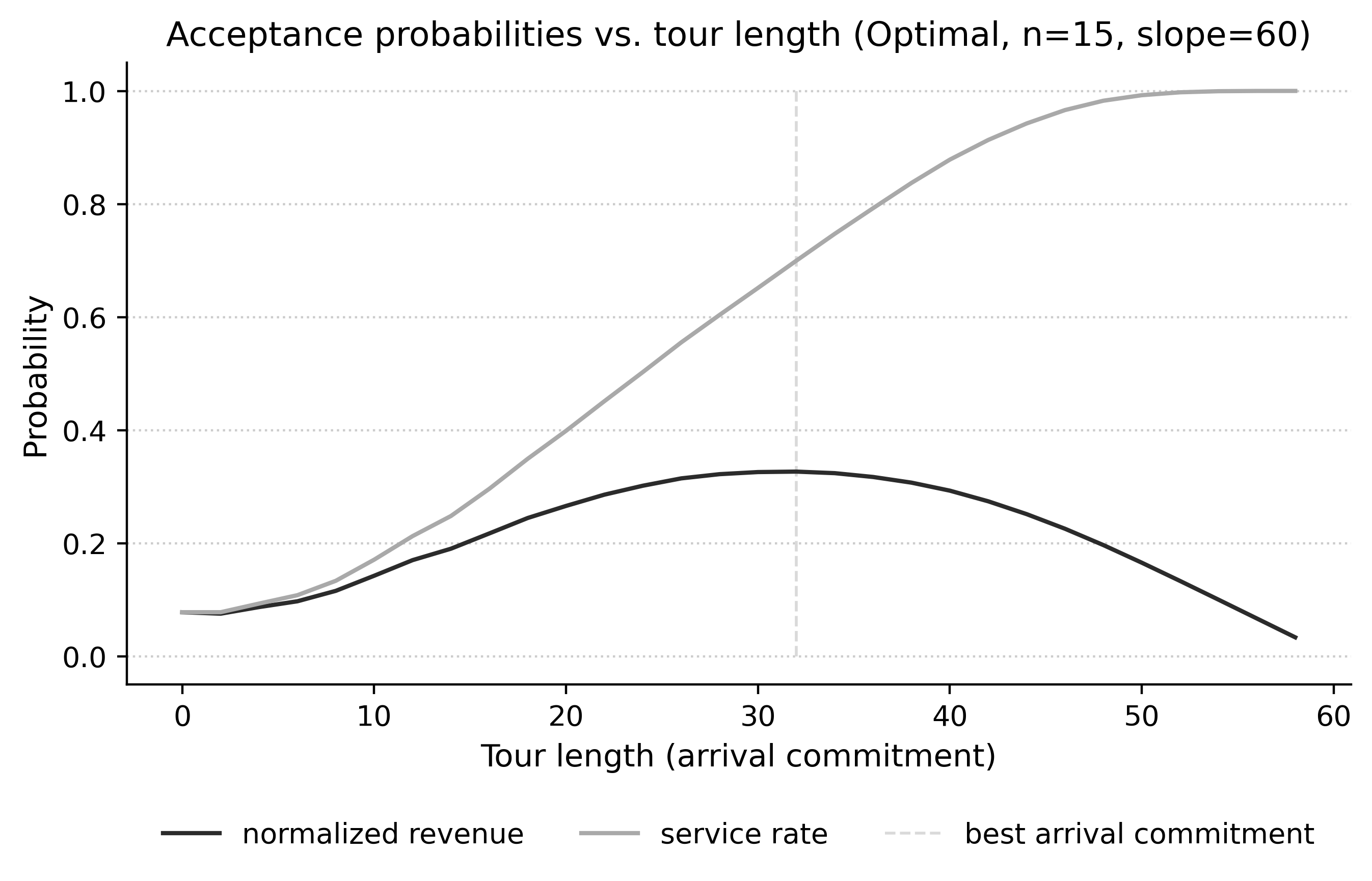}
 	\caption{Normalized expected revenue $R(x)/n$ (dark) and service rate $\mathbb{P}(\text{served}\mid\text{accepted})$ (gray) as functions of the arrival commitment, averaged over all instances with $n = 15$ (slope $= 60$, BSH-\ac{LKH}, reported revenue). The vertical dashed line marks the revenue-maximizing commitment $x^*_\text{LKH}$ of BSH-\ac{LKH}.}
 	\label{fig:acceptance_probas_tour_lengths}
 \end{figure}
 
 Figure~\ref{fig:spatial_distribution} examines two instances with $n = 12$ customers at the same optimal arrival commitment of $x^* = 32$. In each panel, the left plot shows the spatial distribution of customers, where the bubble size reflects how frequently a customer is included in a served subset; the right plot shows the number of subsets of each size that can be served within the commitment, relative to the total number of subsets of that size.
 
 Both instances share the same acceptance probability and subset-size distribution at $x^* = 32$. However, their spatial configurations differ. In the first instance, customers are more clustered but several are located far from the depot. The served subsets are concentrated at smaller subset sizes, and customers far from the depot are rarely included. In the second instance, customers are more evenly distributed and generally closer to the depot. As a result, the served share increases across all subset sizes, and the spatial discrimination of outer customers is less pronounced. Despite the identical arrival commitment and acceptance probability, the attainable revenue differs ($3.91$ vs.\ $4.14$), a gap attributable to the spatial configuration.
 
 \begin{figure}
 	\centering
 	\begin{subfigure}[h]{\textwidth}
 		\centering
 		\includegraphics[width=\textwidth]{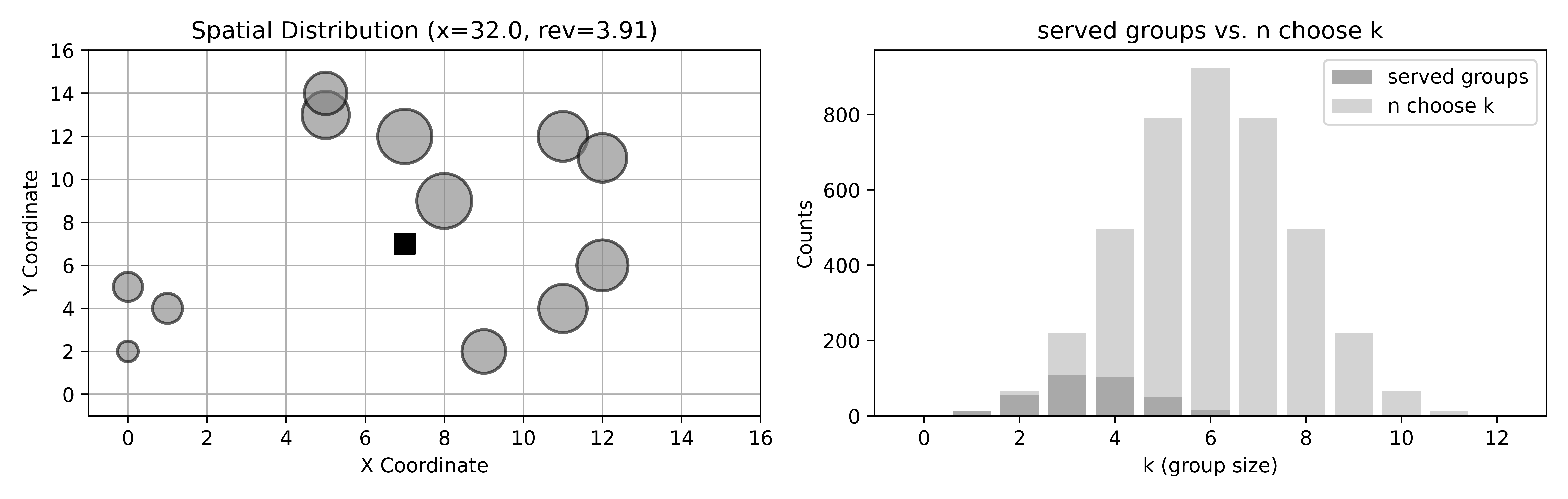}
 		\caption{First instance.}
 	\end{subfigure}
 	\hfill
 	\begin{subfigure}[h]{\textwidth}
 		\centering
 		\includegraphics[width=\textwidth]{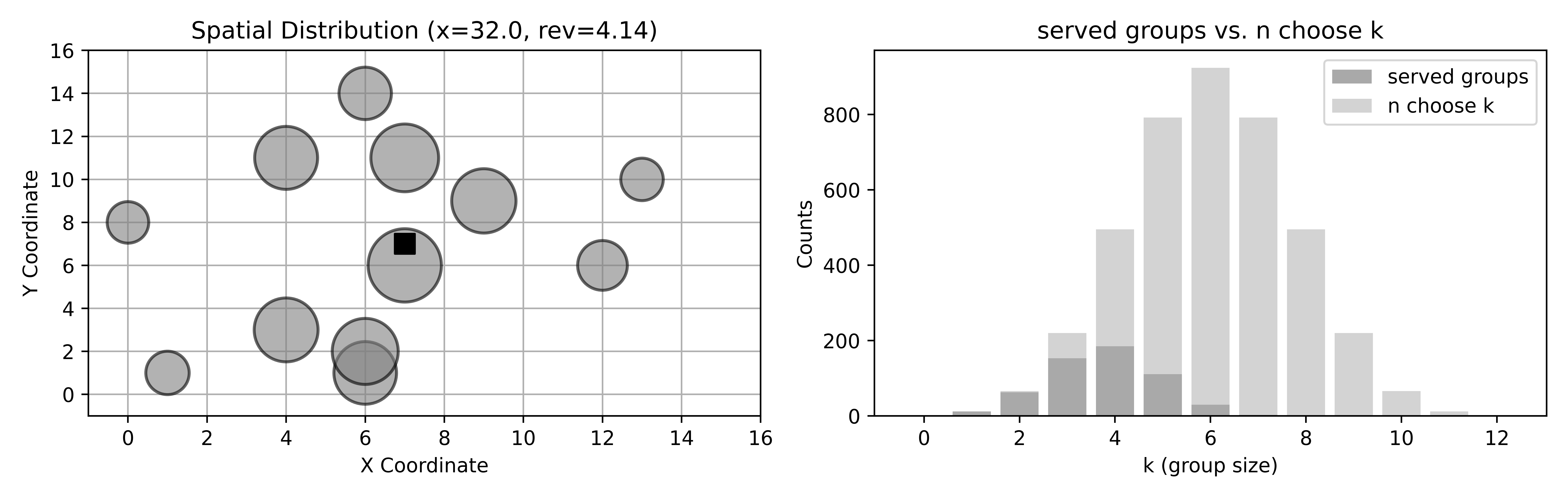}
 		\caption{Second instance.}
 	\end{subfigure}
 	\caption{Spatial distribution of two instances with $n = 12$ customers at the optimal arrival commitment $x^* = 32$. Bubble size reflects how frequently a customer is included in a served subset.}
 	\label{fig:spatial_distribution}
 \end{figure}	
 
\section{Conclusions}
	\label{sec:conc}

	We introduced the \ac{TSPSC}, a two-stage stochastic optimization problem that captures the endogenous interplay between a provider's arrival commitment and stochastic customer demand: a single announced commitment simultaneously shapes each customer's acceptance probability and constrains the routing problem that serves the accepting customers. We showed that an optimal commitment always lies in the finite set of subset tour lengths, and that evaluating the expected revenue of a given commitment is $\textsf{\#P}$-hard for general travel times and customer-specific revenues and acceptance probabilities. Exploiting this discretization, we developed two exact algorithms, a brute-force method with runtime $O(n\,2^{2n})$ and an integrated method with runtime $O(n^2 2^n)$ for homogeneous customers, and, for larger instances, an Adaptive Binomial Sampling Heuristic (BSH).

	Our computational study showed that the heuristic is near-optimal, with a mean relative optimality gap of at most $0.32\%$ across all exactly solvable instance sizes, while scaling to instances with up to 100 customers. It also revealed a robust operational regularity: as the customer base grows, the revenue-maximizing policy settles into a regime in which the provider turns away roughly one in three accepting customers, a rejection rate that then remains stable across instance sizes. A provider who values service reliability beyond its direct revenue contribution can trade a modest amount of expected revenue for a substantially higher service rate by committing to a slightly later arrival time; pushing beyond roughly 80--90\% service, however, becomes disproportionately costly.

	Several directions remain open for future research. On the modeling side, our setting assumes a single vehicle with non-binding capacity, a common destination, and a static one-shot interaction; moreover, the faster exact algorithm, the heuristic, and our experiments assume homogeneous customers. Capacitated and multi-vehicle variants, algorithmic support for heterogeneous customers, and dynamic settings with sequentially arriving requests and personalized commitments are natural extensions. Incorporating an explicit rejection penalty or a service-level constraint would internalize the reliability--revenue tradeoff that our experiments quantify, and jointly optimizing prices and commitments would connect the problem more closely to revenue management. On the theoretical side, the complexity of the first-stage problem remains open, as does whether evaluating the expected revenue stays $\textsf{\#P}$-hard under metric travel times and homogeneous customers. Finally, a formal convergence analysis of the adaptive sampling heuristic, in particular under the re-allocation of the sampling budget across iterations, might yield interesting structural insights.


	\ACKNOWLEDGMENT{For this work, the HPC cluster Hummel-2 at the University of Hamburg was used. The cluster was funded by the Deutsche Forschungsgemeinschaft (DFG, German Research Foundation) -- 498394658.}

		\bibliography{lit}

	\clearpage
	\section*{Appendix A. Parameter Sensitivity: Revenue Across Slope Values}
\setcounter{subsection}{0}
\renewcommand{\thesubsection}{A.\arabic{subsection}}
To assess whether our findings depend on the baseline acceptance function, we vary the slope parameter from 15 to 120 in steps of 15, corresponding to maximum
acceptance windows ranging from 15 minutes to 2 hours. Since BSH-\ac{LKH} demonstrated
strong performance in Section~\ref{subsec:Performance}, the cross-slope analysis is
restricted to this variant. For small instances ($n \leq 19$), BSH-\ac{LKH} results are
compared against the optimal. For large instances ($n \in \{20, 50, 100\}$), only BSH-\ac{LKH}
results are available. All BSH-\ac{LKH} values in this appendix are reported values in
the sense of Section~\ref{subsec:ErrorSources}. In all tables, a dash indicates that
BSH-\ac{LKH} did not terminate within the 12-hour wall-time limit, which occurred only for
$n = 100$ at slope $= 120$.

Two findings emerge from the sensitivity analysis. First, every quantity except the
customer acceptance probability varies monotonically with the slope, and all through the
same mechanism: a more lenient acceptance function makes a later arrival commitment
revenue-optimal, so a larger share of the accepting customers is served within the
committed time. We state this mechanism once here and refer back to it below. Second, the
discrepancies between BSH-\ac{LKH} and the optimal solution trace back to the non-monotone
bias in the arrival commitment analyzed in Section~\ref{sec:estimation_error}: the
commitment is overestimated at low slope values, most accurate near the baseline slope of
60, and underestimated at high slope values for the larger instances, so the downstream
quantities are also most accurate near the baseline.

\subsection{Expected Revenue}
As shown in Table~\ref{tab:revenue_slope}, expected revenue increases
monotonically with the slope across all instance sizes, driven by the mechanism stated
above. Revenue is also higher for larger customer sets, which is intuitive given that
more potential demand is available to be served. The dash at $n = 100$, slope $= 120$
reflects the computational cost of repeated \ac{LKH} calls for the larger customer
subsets that predominate at high slope values rather than a fundamental algorithmic
limitation.

A downward bias in BSH-\ac{LKH} revenue relative to the optimal is observed at slope
$= 15$ for every instance size and at slope $= 30$ for $n \leq 12$. In this regime,
the revenue-maximizing commitment $x^*$ is short, rendering the
vast majority of customer subsets infeasible, so nearly all revenue must be estimated
through the \textsc{PartialRev} procedure. As discussed in Section~\ref{sec:estimation_error},
\textsc{PartialRev} underestimates the partial-service revenue of no-instances when the
served shares $\hat{f}_k(x)$ are small, which is precisely the regime induced by a short
commitment. This underestimation is the primary source of the downward bias, and
diminishes as the slope increases and a larger fraction of customer subsets become
yes-instances.

\begin{table}[t]
\centering
\small
\begin{tabular}{c l rrrrrrrr}
\toprule
& & \multicolumn{8}{c}{Slope} \\
\cmidrule(lr){3-10}
$n$ & Algorithm & 15 & 30 & 45 & 60 & 75 & 90 & 105 & 120 \\
\midrule
\multirow{2}{*}{4} & Optimal & 0.54 & 1.14 & 1.59 & 1.95 & 2.24 & 2.47 & 2.67 & 2.83 \\
 & BSH-LKH & 0.48 & 1.09 & 1.55 & 1.92 & 2.21 & 2.46 & 2.66 & 2.82 \\
\midrule
\multirow{2}{*}{8} & Optimal & 0.81 & 1.71 & 2.49 & 3.15 & 3.70 & 4.17 & 4.56 & 4.89 \\
 & BSH-LKH & 0.62 & 1.63 & 2.46 & 3.12 & 3.66 & 4.12 & 4.51 & 4.85 \\
\midrule
\multirow{2}{*}{10} & Optimal & 0.92 & 1.97 & 2.92 & 3.69 & 4.37 & 4.97 & 5.46 & 5.89 \\
 & BSH-LKH & 0.72 & 1.91 & 2.93 & 3.72 & 4.39 & 4.96 & 5.44 & 5.85 \\
\midrule
\multirow{2}{*}{12} & Optimal & 1.04 & 2.23 & 3.33 & 4.25 & 5.04 & 5.72 & 6.31 & 6.82 \\
 & BSH-LKH & 0.81 & 2.20 & 3.40 & 4.36 & 5.13 & 5.78 & 6.33 & 6.81 \\
\midrule
\multirow{2}{*}{15} & Optimal & 1.18 & 2.54 & 3.86 & 4.97 & 5.94 & 6.79 & 7.53 & 8.16 \\
 & BSH-LKH & 0.91 & 2.57 & 4.03 & 5.21 & 6.15 & 6.96 & 7.65 & 8.24 \\
\midrule
20 & BSH-LKH & 0.98 & 3.04 & 4.91 & 6.46 & 7.75 & 8.82 & 9.74 & 10.53 \\
\midrule
50 & BSH-LKH & 1.10 & 4.61 & 8.20 & 11.65 & 14.71 & 17.49 & 19.86 & 21.95 \\
\midrule
100 & BSH-LKH & 1.30 & 5.33 & 10.41 & 15.65 & 20.74 & 25.64 & 30.13 & -- \\
\bottomrule
\end{tabular}
\caption{Mean expected revenue by instance size and slope, optimal vs.\ BSH-\ac{LKH}.}
\label{tab:revenue_slope}
\end{table}

\subsection{Arrival Commitment}
Table~\ref{tab:commitment_slope} shows the mean arrival commitment for different slopes. 
The commitment increases monotonically with the slope, as a higher slope
extends the range over which customers accept the offer, making later commitments
revenue-optimal.

BSH-\ac{LKH} overestimates the arrival commitment at low slope values, achieves closest agreement 
with the optimal near slope $= 60$, and, for $n \geq 10$, transitions to underestimation at high slope values. 
Two competing error sources inherent to the heuristic explain this non-monotone bias
pattern. At low slope values, the revenue-maximizing commitment is short, corresponding to
small customer subsets for which \ac{LKH} computes near-optimal tour lengths. The dominant
source of error in this regime is therefore not the tour-length oracle but the
underestimation of expected revenue caused by \textsc{PartialRev}
(see Table~\ref{tab:revenue_slope}). What matters for the commitment decision is not
how large this underestimation is, but how it changes with $x$. As noted above, the \textsc{PartialRev} bias is most severe for short commitments and
shrinks as $x$ grows. Around a short optimum $x^*$, the estimated revenue thus continues
to rise with $x$ because its downward bias shrinks faster than the true revenue falls.
The estimated revenue therefore peaks to the right of $x^*$, and BSH-\ac{LKH} selects a
commitment that exceeds $x^*$. 

At high slope values, BSH-\ac{LKH} tends to underestimate the optimal arrival commitment,
particularly for larger instance sizes. The
root cause is again a bias in the estimated marginal gain of extending the commitment, but
here it works in the opposite direction. In this regime, the revenue-maximizing commitment
$x^*$ is long, corresponding to larger customer subsets for which \ac{LKH} tends to
overestimate tour lengths. Some subsets that truly fit within the commitment are therefore 
classified as infeasible. Extending the commitment thus reveals fewer new yes-instances in
BSH-\ac{LKH}'s view than in reality, so the estimated marginal gain from a later commitment
is smaller than the true one, while the drop in acceptance probability is unchanged. 
The estimated revenue curve consequently stops climbing too early and peaks
at $x^*_\text{LKH} < x^*$: BSH-\ac{LKH} selects a commitment that falls short of the
true optimum.

\begin{table}[t]
\centering
\small
\begin{tabular}{c l rrrrrrrr}
\toprule
& & \multicolumn{8}{c}{Slope} \\
\cmidrule(lr){3-10}
$n$ & Algorithm & 15 & 30 & 45 & 60 & 75 & 90 & 105 & 120 \\
\midrule
\multirow{2}{*}{4} & Optimal & 8.9 & 15.2 & 20.9 & 25.5 & 28.4 & 32.0 & 33.8 & 35.0 \\
 & BSH-LKH & 9.3 & 16.0 & 21.3 & 25.9 & 29.9 & 32.6 & 34.6 & 35.1 \\
\midrule
\multirow{2}{*}{8} & Optimal & 7.9 & 17.3 & 23.7 & 29.0 & 33.3 & 36.8 & 39.7 & 41.9 \\
 & BSH-LKH & 9.7 & 18.7 & 24.5 & 29.2 & 33.0 & 37.0 & 40.5 & 43.1 \\
\midrule
\multirow{2}{*}{10} & Optimal & 8.2 & 17.1 & 23.9 & 29.4 & 35.0 & 39.2 & 41.7 & 44.6 \\
 & BSH-LKH & 10.3 & 19.0 & 24.8 & 29.9 & 34.1 & 38.4 & 40.9 & 43.9 \\
\midrule
\multirow{2}{*}{12} & Optimal & 8.0 & 17.8 & 24.2 & 30.1 & 35.2 & 39.8 & 43.3 & 46.1 \\
 & BSH-LKH & 10.6 & 19.4 & 25.3 & 30.2 & 34.4 & 38.0 & 41.5 & 44.3 \\
\midrule
\multirow{2}{*}{15} & Optimal & 8.2 & 17.9 & 25.2 & 31.1 & 36.6 & 41.4 & 45.2 & 48.3 \\
 & BSH-LKH & 11.1 & 19.9 & 26.1 & 31.2 & 35.6 & 39.4 & 43.0 & 46.0 \\
\midrule
20 & BSH-LKH & 12.0 & 20.9 & 27.7 & 33.4 & 37.7 & 41.9 & 45.3 & 48.4 \\
\midrule
50 & BSH-LKH & 12.6 & 23.7 & 33.3 & 41.2 & 47.6 & 52.8 & 57.3 & 61.0 \\
\midrule
100 & BSH-LKH & 13.6 & 25.6 & 37.1 & 47.0 & 55.5 & 62.8 & 69.1 & -- \\
\bottomrule
\end{tabular}
\caption{Mean arrival commitment by instance size and slope, optimal vs.\ BSH-\ac{LKH}.}
\label{tab:commitment_slope}
\end{table}

\subsection{Commitment-Tour Ratio}
Table~\ref{tab:comm_tour_slope_grid} reports the mean commitment-tour ratio as a
function of the slope. The ratio increases monotonically with the slope across all
instance sizes: the later commitments that become revenue-optimal at higher slopes
cover a progressively larger fraction of the full-service tour. The deviation of
BSH-\ac{LKH} from the optimal broadly follows the arrival commitment bias described
above: the ratio tends to be overestimated at low slope values and underestimated at
high slope values, with the discrepancy widening with $n$. For BSH-\ac{LKH}, the
full-service tour length in the denominator is itself computed from the LKH tour,
so its tendency to overestimate tour length inflates the denominator and can depress
the ratio even when the commitment (see Table~\ref{tab:commitment_slope}) in the 
numerator is overestimated.


\begin{table}[t]
\centering
\small
\begin{tabular}{c l rrrrrrrr}
	\toprule
	& & \multicolumn{8}{c}{Slope} \\
	\cmidrule(lr){3-10}
	$n$ & Algorithm & 15 & 30 & 45 & 60 & 75 & 90 & 105 & 120 \\
	\midrule
	\multirow{2}{*}{4} & Optimal & 0.260 & 0.444 & 0.623 & 0.751 & 0.833 & 0.928 & 0.972 & 0.997 \\
	& BSH-LKH & 0.273 & 0.472 & 0.633 & 0.760 & 0.870 & 0.939 & 0.984 & 0.996 \\
	\midrule
	\multirow{2}{*}{8} & Optimal & 0.163 & 0.356 & 0.488 & 0.598 & 0.683 & 0.755 & 0.814 & 0.858 \\
	& BSH-LKH & 0.198 & 0.378 & 0.494 & 0.589 & 0.665 & 0.744 & 0.816 & 0.865 \\
	\midrule
	\multirow{2}{*}{10} & Optimal & 0.156 & 0.325 & 0.452 & 0.556 & 0.661 & 0.741 & 0.787 & 0.841 \\
	& BSH-LKH & 0.189 & 0.346 & 0.451 & 0.543 & 0.621 & 0.699 & 0.743 & 0.797 \\
	\midrule
	\multirow{2}{*}{12} & Optimal & 0.140 & 0.310 & 0.419 & 0.522 & 0.611 & 0.691 & 0.750 & 0.798 \\
	& BSH-LKH & 0.175 & 0.319 & 0.415 & 0.497 & 0.565 & 0.625 & 0.683 & 0.727 \\
	\midrule
	\multirow{2}{*}{15} & Optimal & 0.129 & 0.281 & 0.396 & 0.489 & 0.576 & 0.652 & 0.711 & 0.759 \\
	& BSH-LKH & 0.164 & 0.291 & 0.382 & 0.457 & 0.521 & 0.577 & 0.629 & 0.672 \\
	\midrule
	\multirow{1}{*}{20} & BSH-LKH & 0.156 & 0.272 & 0.361 & 0.432 & 0.487 & 0.538 & 0.580 & 0.616 \\
	\midrule
	\multirow{1}{*}{50} & BSH-LKH & 0.113 & 0.211 & 0.295 & 0.366 & 0.423 & 0.469 & 0.510 & 0.544 \\
	\midrule
	\multirow{1}{*}{100} & BSH-LKH & 0.086 & 0.161 & 0.233 & 0.296 & 0.350 & 0.394 & 0.435 & -- \\
	\bottomrule
\end{tabular}
\caption{Mean commitment-tour ratio by instance size and slope, optimal vs.\ BSH-\ac{LKH}.}
\label{tab:comm_tour_slope_grid}
\end{table}

\subsection{Customer Acceptance Probability}
Table~\ref{tab:cust_acc_slope_grid} reports the mean customer acceptance probability at
the revenue-maximizing commitment, and Figure~\ref{fig:acceptanceProbasSlope} plots it as
a function of the slope for the small instances (Figure~\ref{fig:acceptanceProbasSlopeSmall};
optimal solid, BSH-\ac{LKH} dashed) and the large instances
(Figure~\ref{fig:acceptanceProbasSlopeLarge}). For the instances with $n \geq 8$ for which
the optimum is available, the optimal acceptance probability is non-monotone: it declines
from slope $= 15$ to slope $= 30$ before increasing monotonically for larger slopes. To
understand this, recall that $p(x) = \max(0, 1 - x/\text{slope})$, so the acceptance
probability depends not on $x$ alone but on the ratio $x/\text{slope}$. As the slope
increases the provider commits to a later arrival time, but $x$ need not grow in proportion
to the slope. From slope $= 15$ to slope $= 30$, the optimal commitment grows faster than
the slope, so the ratio $x/\text{slope}$ rises and the acceptance probability falls. Beyond
slope $= 30$, the commitment keeps increasing but more slowly than the slope, so the ratio
declines steadily and the acceptance probability rises.

A related phenomenon is the pronounced gap between BSH-\ac{LKH} and the optimal
acceptance probability at low slope values. At slope $\in \{15, 30\}$, the
overestimation of the arrival commitment by BSH-\ac{LKH} translates into a substantial
reduction in acceptance probability, because the sensitivity of $p$ to changes in
$x$ is $1/\text{slope}$, which is large when the slope is small. This accounts for
the divergence between the dashed and solid lines at low slopes in
Figure~\ref{fig:acceptanceProbasSlope}. As the slope increases, this sensitivity
diminishes, and the underestimation of the arrival commitment observed for slope
values of 75 and above has a progressively smaller effect on the resulting acceptance
probability. The gap between BSH-\ac{LKH} and the optimal is therefore smallest around
the baseline slope of 60. For $n \geq 10$ and slope $\geq 75$, BSH-\ac{LKH} even exceeds
the optimal acceptance probability. This is not an improvement but the flip side of the
underestimated commitment: the heuristic offers an earlier arrival time than optimal,
which more customers accept, at the cost of serving a smaller share of them.

\begin{figure}[h]
	\centering
	\begin{subfigure}{0.48\textwidth}
		\centering
		\includegraphics[width=\linewidth]{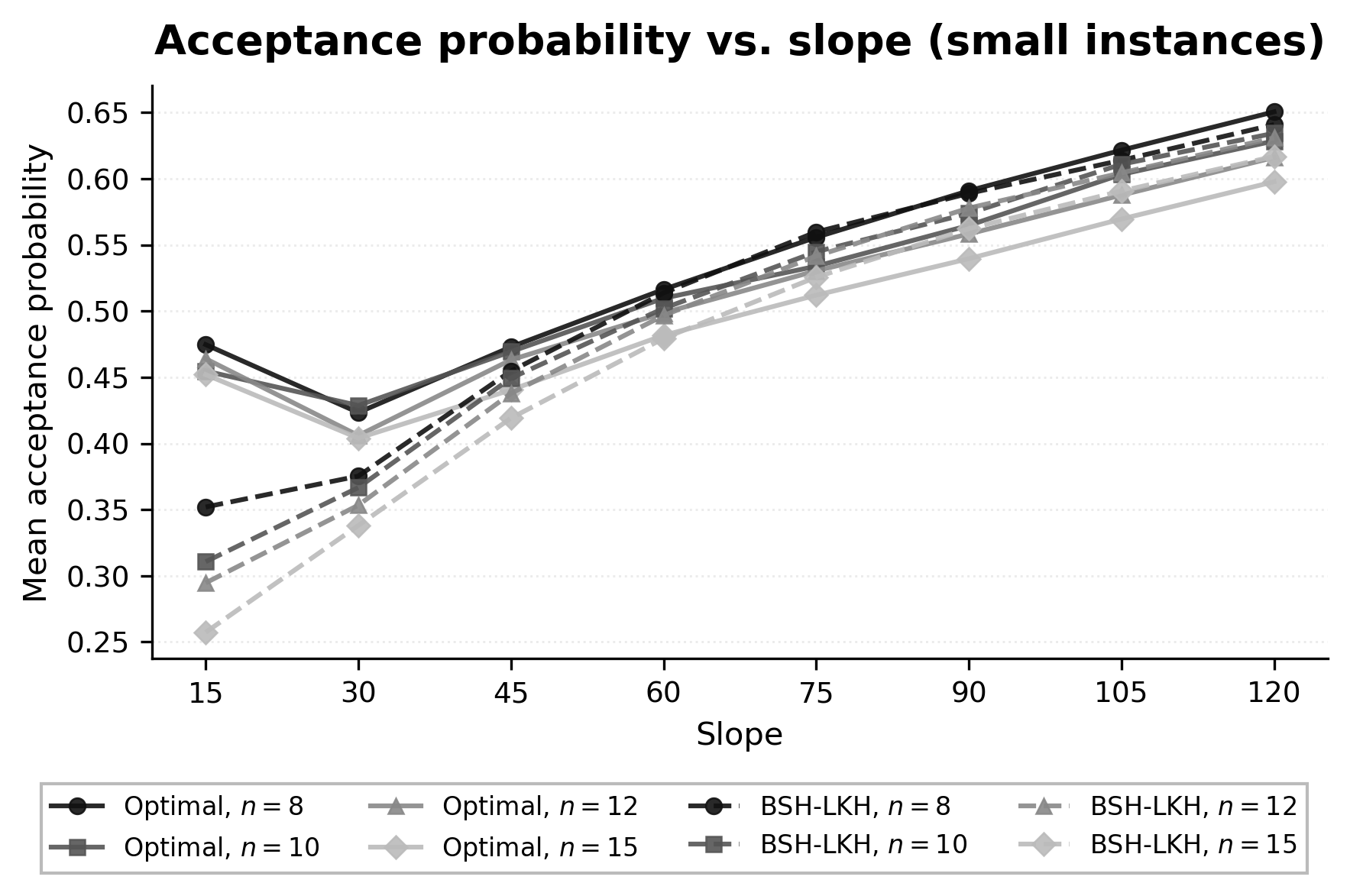}
		\caption{Small instances.}
		\label{fig:acceptanceProbasSlopeSmall}
	\end{subfigure}
	\hfill
	\begin{subfigure}{0.48\textwidth}
		\centering
		\includegraphics[width=\linewidth]{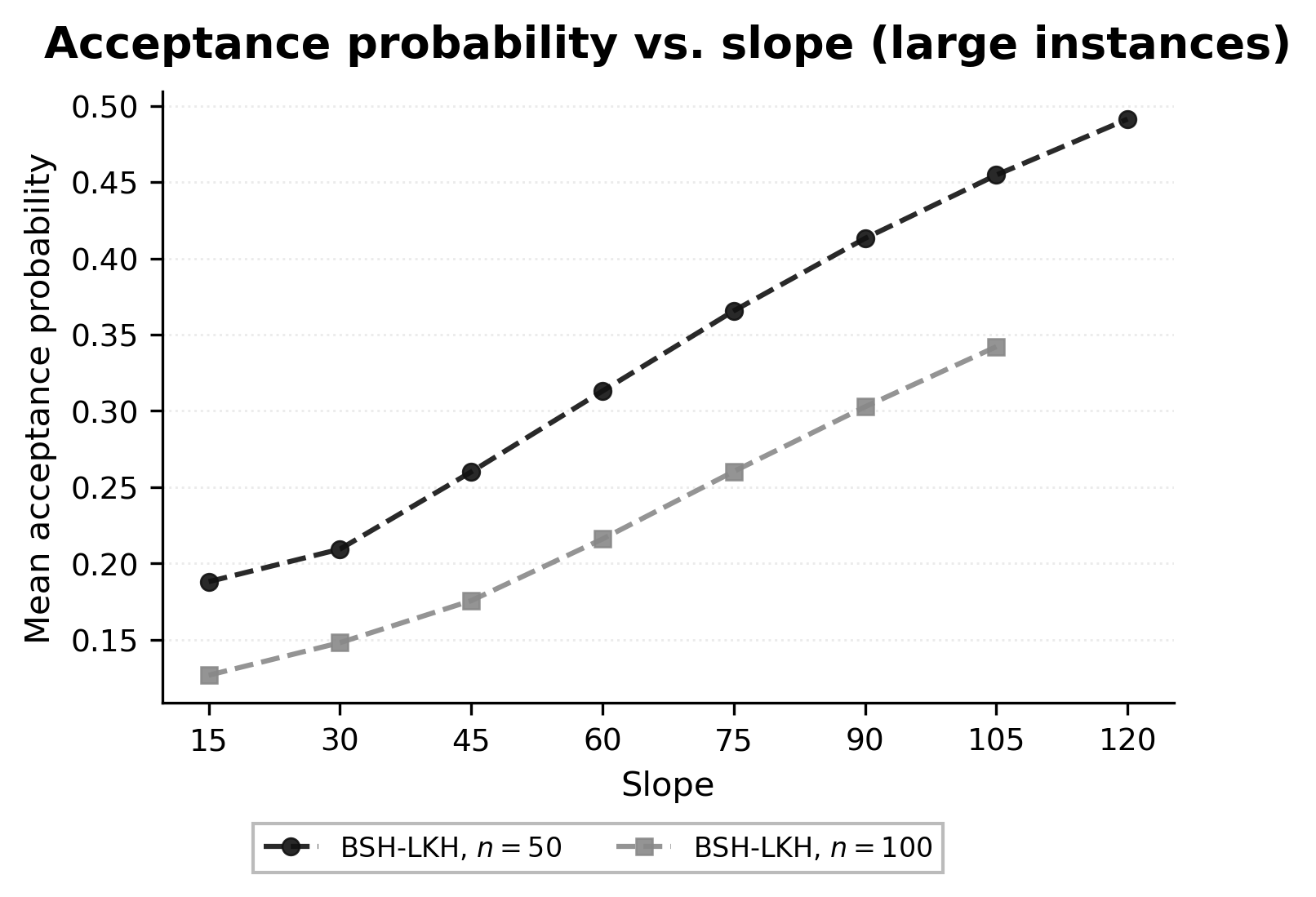}
		\caption{Large instances.}
		\label{fig:acceptanceProbasSlopeLarge}
	\end{subfigure}
	\caption{Mean customer acceptance probability at the revenue-maximizing commitment as a function of the slope.}
	\label{fig:acceptanceProbasSlope}
\end{figure}


\begin{table}[t]
\centering
\small
\begin{tabular}{c l rrrrrrrr}
	\toprule
	& & \multicolumn{8}{c}{Slope} \\
	\cmidrule(lr){3-10}
	$n$ & Algorithm & 15 & 30 & 45 & 60 & 75 & 90 & 105 & 120 \\
	\midrule
	\multirow{2}{*}{4} & Optimal & 0.411 & 0.497 & 0.535 & 0.574 & 0.621 & 0.644 & 0.678 & 0.708 \\
	& BSH-LKH & 0.384 & 0.467 & 0.526 & 0.568 & 0.602 & 0.638 & 0.671 & 0.707 \\
	\midrule
	\multirow{2}{*}{8} & Optimal & 0.475 & 0.423 & 0.473 & 0.516 & 0.556 & 0.591 & 0.622 & 0.651 \\
	& BSH-LKH & 0.352 & 0.375 & 0.455 & 0.514 & 0.560 & 0.589 & 0.614 & 0.641 \\
	\midrule
	\multirow{2}{*}{10} & Optimal & 0.455 & 0.429 & 0.469 & 0.510 & 0.534 & 0.565 & 0.603 & 0.628 \\
	& BSH-LKH & 0.311 & 0.367 & 0.449 & 0.502 & 0.545 & 0.574 & 0.611 & 0.634 \\
	\midrule
	\multirow{2}{*}{12} & Optimal & 0.464 & 0.406 & 0.463 & 0.499 & 0.530 & 0.558 & 0.588 & 0.616 \\
	& BSH-LKH & 0.295 & 0.353 & 0.438 & 0.497 & 0.542 & 0.578 & 0.604 & 0.631 \\
	\midrule
	\multirow{2}{*}{15} & Optimal & 0.452 & 0.404 & 0.440 & 0.482 & 0.512 & 0.540 & 0.570 & 0.598 \\
	& BSH-LKH & 0.257 & 0.338 & 0.419 & 0.479 & 0.526 & 0.562 & 0.591 & 0.617 \\
	\midrule
	\multirow{1}{*}{20} & BSH-LKH & 0.214 & 0.303 & 0.384 & 0.444 & 0.497 & 0.534 & 0.568 & 0.597 \\
	\midrule
	\multirow{1}{*}{50} & BSH-LKH & 0.188 & 0.209 & 0.260 & 0.313 & 0.366 & 0.413 & 0.455 & 0.491 \\
	\midrule
	\multirow{1}{*}{100} & BSH-LKH & 0.127 & 0.148 & 0.176 & 0.216 & 0.260 & 0.303 & 0.342 & -- \\
	\bottomrule
\end{tabular}
\caption{Mean customer acceptance probability at the revenue-maximizing commitment by instance size and slope, optimal vs.\ BSH-\ac{LKH}.}
\label{tab:cust_acc_slope_grid}
\end{table}

\subsection{Normalized Revenue}
Table~\ref{tab:norm_revenue_slope_grid} reports the normalized revenue by
instance size and slope. In contrast to
the customer acceptance probability, the normalized revenue increases
monotonically with the slope for every instance size, with no dip at slope $= 30$, again
driven by the mechanism stated above. At a fixed slope, it decreases in $n$: for
slope $= 60$ from $39.1\%$ at $n = 8$ to $15.6\%$ at $n = 100$ for BSH-\ac{LKH}. As the
customer base grows, the commitment covers a smaller fraction of the full-service tour
(Table~\ref{tab:comm_tour_slope_grid}), so a smaller share of accepting customers is
served.

Across the small instances, BSH-\ac{LKH} tracks the optimal closely, with an absolute gap
of at most $2.3$ percentage points over the entire slope range. The discrepancy is largest
at slope $= 15$, where BSH-\ac{LKH} underestimates the normalized revenue for every $n$
(e.g., $6.1\%$ vs.\ $7.9\%$ at $n = 15$), and it diminishes markedly from slope $= 45$
onward.
This low-slope underestimation is consistent with the \ac{LKH} tour-length overestimation
established in Section~\ref{sec:estimation_error}: at short commitments, feasibility is
assessed against inflated tour lengths, so a larger fraction of subsets is classified as
infeasible and the served share is understated.

\begin{table}[t]
\centering
\small
\begin{tabular}{c l rrrrrrrr}
	\toprule
	& & \multicolumn{8}{c}{Slope} \\
	\cmidrule(lr){3-10}
	$n$ & Algorithm & 15 & 30 & 45 & 60 & 75 & 90 & 105 & 120 \\
	\midrule
	\multirow{2}{*}{4} & Optimal & 0.136 & 0.285 & 0.398 & 0.488 & 0.559 & 0.617 & 0.667 & 0.707 \\
	& BSH-LKH & 0.120 & 0.271 & 0.388 & 0.481 & 0.554 & 0.615 & 0.665 & 0.706 \\
	\midrule
	\multirow{2}{*}{8} & Optimal & 0.101 & 0.213 & 0.311 & 0.394 & 0.463 & 0.521 & 0.570 & 0.611 \\
	& BSH-LKH & 0.078 & 0.203 & 0.308 & 0.391 & 0.458 & 0.515 & 0.564 & 0.606 \\
	\midrule
	\multirow{2}{*}{10} & Optimal & 0.092 & 0.197 & 0.292 & 0.369 & 0.437 & 0.497 & 0.546 & 0.589 \\
	& BSH-LKH & 0.072 & 0.191 & 0.293 & 0.372 & 0.439 & 0.496 & 0.544 & 0.585 \\
	\midrule
	\multirow{2}{*}{12} & Optimal & 0.086 & 0.186 & 0.277 & 0.354 & 0.420 & 0.477 & 0.526 & 0.568 \\
	& BSH-LKH & 0.068 & 0.184 & 0.283 & 0.363 & 0.428 & 0.482 & 0.528 & 0.568 \\
	\midrule
	\multirow{2}{*}{15} & Optimal & 0.079 & 0.170 & 0.258 & 0.331 & 0.396 & 0.453 & 0.502 & 0.544 \\
	& BSH-LKH & 0.061 & 0.171 & 0.269 & 0.347 & 0.410 & 0.464 & 0.510 & 0.549 \\
	\midrule
	\multirow{1}{*}{20} & BSH-LKH & 0.049 & 0.152 & 0.245 & 0.323 & 0.387 & 0.441 & 0.487 & 0.527 \\
	\midrule
	\multirow{1}{*}{50} & BSH-LKH & 0.022 & 0.092 & 0.164 & 0.233 & 0.294 & 0.350 & 0.397 & 0.439 \\
	\midrule
	\multirow{1}{*}{100} & BSH-LKH & 0.013 & 0.053 & 0.104 & 0.156 & 0.207 & 0.256 & 0.301 & -- \\
	\bottomrule
\end{tabular}
\caption{Mean normalized revenue by instance size and slope, optimal vs.\ BSH-\ac{LKH}.}
\label{tab:norm_revenue_slope_grid}
\end{table}

\subsection{Service Rate}
Table~\ref{tab:service_rate_slope_grid} reports the service rate
(Section~\ref{subsec:Managerial}) by instance size and slope. In line with the
normalized revenue, the service rate increases monotonically with the slope for every
instance size, again driven by the mechanism stated above. The effect is
pronounced: for $n = 100$ the service rate rises from $10.3\%$ at slope $= 15$ to $88.1\%$
at slope $= 105$. Low slopes thus correspond to impatient markets in which the provider
must turn away a large fraction of accepting customers, whereas at high slopes almost every
accepting customer is served. At the baseline slope $= 60$, the reported service rate
settles in a narrow band around $72$--$74\%$ for the larger instances. As discussed in
Section~\ref{subsec:Managerial}, the \textsc{PartialRev} overestimation inflates this
reported rate; the true rate is closer to $69\%$, so the revenue-maximizing provider still
turns away roughly one in three accepting customers.

The bias of BSH-\ac{LKH} relative to the optimal follows the arrival commitment bias. The
service rate is overestimated at low slope values (slope $\leq 45$, for every $n$), agrees
closely near slope $= 60$, and is underestimated at high slope values (slope $\geq 90$) for $n \geq 10$. 
The service rate is the ratio of the normalized revenue
(Table~\ref{tab:norm_revenue_slope_grid}) to the customer acceptance probability
(Table~\ref{tab:cust_acc_slope_grid}). At low slopes, the overestimated commitment
depresses the denominator more than the numerator, so the ratio is inflated; at high
slopes, the underestimated commitment has the opposite effect, and the ratio is deflated.
As with the other quantities, the discrepancy narrows around
the baseline slope, where the commitment bias changes sign.

\begin{table}[t]
	\centering
	\small
	\begin{tabular}{c l rrrrrrrr}
		\toprule
		& & \multicolumn{8}{c}{Slope} \\
		\cmidrule(lr){3-10}
		$n$ & Algorithm & 15 & 30 & 45 & 60 & 75 & 90 & 105 & 120 \\
		\midrule
		\multirow{2}{*}{4} & Optimal & 0.331 & 0.573 & 0.745 & 0.850 & 0.900 & 0.958 & 0.984 & 0.998 \\
		& BSH-LKH & 0.313 & 0.581 & 0.739 & 0.847 & 0.920 & 0.964 & 0.992 & 0.998 \\
		\midrule
		\multirow{2}{*}{8} & Optimal & 0.213 & 0.504 & 0.659 & 0.764 & 0.833 & 0.882 & 0.917 & 0.940 \\
		& BSH-LKH & 0.220 & 0.541 & 0.677 & 0.760 & 0.818 & 0.875 & 0.918 & 0.945 \\
		\midrule
		\multirow{2}{*}{10} & Optimal & 0.202 & 0.459 & 0.621 & 0.725 & 0.819 & 0.880 & 0.906 & 0.937 \\
		& BSH-LKH & 0.231 & 0.520 & 0.652 & 0.741 & 0.806 & 0.864 & 0.890 & 0.922 \\
		\midrule
		\multirow{2}{*}{12} & Optimal & 0.186 & 0.458 & 0.598 & 0.710 & 0.792 & 0.855 & 0.896 & 0.923 \\
		& BSH-LKH & 0.229 & 0.519 & 0.647 & 0.731 & 0.790 & 0.834 & 0.873 & 0.900 \\
		\midrule
		\multirow{2}{*}{15} & Optimal & 0.175 & 0.420 & 0.585 & 0.688 & 0.773 & 0.839 & 0.881 & 0.911 \\
		& BSH-LKH & 0.237 & 0.506 & 0.642 & 0.724 & 0.781 & 0.825 & 0.863 & 0.890 \\
		\midrule
		\multirow{1}{*}{20} & BSH-LKH & 0.229 & 0.502 & 0.639 & 0.728 & 0.779 & 0.825 & 0.857 & 0.882 \\
		\midrule
		\multirow{1}{*}{50} & BSH-LKH & 0.117 & 0.440 & 0.631 & 0.744 & 0.805 & 0.847 & 0.873 & 0.894 \\
		\midrule
		\multirow{1}{*}{100} & BSH-LKH & 0.103 & 0.360 & 0.593 & 0.724 & 0.797 & 0.847 & 0.881 & -- \\
		\bottomrule
	\end{tabular}
	\caption{Mean service rate by instance size and slope, optimal vs.\ BSH-\ac{LKH}.}
	\label{tab:service_rate_slope_grid}
\end{table}


\end{document}